\documentclass[10pt]{article}
\usepackage{amsmath,amsthm,amssymb,float}
\usepackage[all]{xy}

\usepackage{afterpage}
\usepackage{url}
\usepackage{todonotes}
\usepackage[margin=1.2in]{geometry}
\usepackage[T1]{fontenc}
\usepackage{mathpazo}
\usepackage{bm}
\usepackage{tabto}

\usepackage[font=small]{caption}
\usepackage[font=footnotesize]{subcaption}

\usepackage{graphicx}

\floatstyle{plain}
\restylefloat{figure}

\usepackage{setspace}

\newcommand{\R}{\mathbb{R}}

\newcommand{\N}{\mathbb{N}}

\newcommand{\eps}{\epsilon}

\newcommand{\se}[1]{\mathrm{\, E}{#1}}

\newcommand{\thewebsite}{\mbox{\url{gitlab.com/s_ament/qastable}}}

\newcommand{\vct}[1]{\bm{\mathsf{#1}}}

\newcommand{\mtx}[1]{\bm{\mathsf{#1}}}

\newcommand{\E}{\operatorname{E}}
\renewcommand{\Re}{\operatorname{Re}}
\renewcommand{\Im}{\operatorname{Im}}

\newtheorem{lemma}{Lemma}

\theoremstyle{definition}

\theoremstyle{remark}
\newtheorem{remark}{Remark}

\usepackage{titlesec}
\usepackage{lastpage}
\titleformat{\section}
  {\bf\large}{\thesection}{1em}{}
\titleformat{\subsection}
  {\normalfont\bf}{\thesubsection}{1em}{}
\titleformat{\subsubsection}
  {\normalfont\bf}{\thesubsubsection}{1em}{}

\usepackage{fancyhdr}
\numberwithin{equation}{section}

\UseRawInputEncoding

\begin{document}

\title{Accurate and efficient numerical calculation of stable densities
  via optimized quadrature and asymptotics}
  
\author{Sebastian Ament\footnote{\sf Department of Computer Science,
	Cornell University, New York,
    NY. E-mail: {\tt ament@cs.cornell.edu}.} \, and
  Michael O'Neil\footnote{\sf Department of Mathematics, Courant
    Institute and Tandon School of Engineering, New York
    University, New York, NY. E-mail: {\tt
      oneil@cims.nyu.edu}.  Research partially supported by AIG-NYU
    Award \#A15-0098-001. } }

\date{\today}

\maketitle

\begin{abstract}
  Stable distributions are an important class of infinitely-divisible
  probability distributions, of which two special cases are the Cauchy
  distribution and the normal distribution.  Aside from a few special
  cases, the density function for stable distributions has no known
  analytic form, and is expressible only through the variate's
  characteristic function or other integral forms. 
  In this paper we present numerical schemes for evaluating
  the density function for stable distributions, its gradient,
  and distribution function in various parameter
  regimes of interest, some of which had no pre-existing efficient
  method for their computation. The novel evaluation schemes consist
  of optimized generalized Gaussian quadrature rules for integral
  representations of the density function, complemented by 
  asymptotic
  expansions near various values of the shape and argument
  parameters. We report
  several numerical examples illustrating the efficiency of our
  methods.
  The resulting code has been made available online.
\end{abstract}

\onehalfspacing

\section{Introduction}
\label{sec_intro}

Continuous random variables that follow {\em
  stable laws} arise frequently in physics~\cite{kolbig_1984}, finance
and economics~\cite{mittnik_1993, nolan_2003}, electrical
engineering~\cite{nikias_1995}, and many other fields of the 
natural and social sciences. Certain sub-classes of 
these distributions are also referred to as
$\alpha$-stable, $a$-stable, stable Paretian distributions, or L\'evy
alpha-stable distributions. Going forward, we will merely refer to
them as stable distributions.  The defining characteristic of random
variables that follow stable laws is that the sum of two
independent copies follows the {\em same} scaled and translated
distribution~\cite{nolan_2015}.  For example, if $X_1$ and $X_2$ are
independent, identically distributed (iid) stable random variables,
then {\em in distribution}
\begin{equation}\label{eq_sum}
X_1 + X_2 \sim aX + b,
\end{equation}
where $X$ has the same distribution as each $X_\ell$. Several discrete
random variables, such as those following Poisson distributions, also
obey this {\em stability-of-sums} law, but we restrict our attention to 
continuous distributions. 
Modeling with stable distribution has several advantages. For example,
even though in general they do not have finite variances, 
they are closed under sums
and satisfy a generalized-type of Central Limit
Theorem~\cite{nikias_1995,zolotarev_1986}. This property
is directly related to the fact that they have
tails which are heavier than those of normal random variables.
For this reason, these distributions are  useful in describing many 
real-world data sets from finance, physics, and chemistry.
On the other hand, computing with stable distributions requires more
sophistication than does, for example, computing with normal
distributions. When modeling with multivariate normal distributions,
all of the relevant calculations (in, for example, likelihood
evaluation) are linear-algebraic in nature:
matrix inversion, determinant calculation, eigenvalue computation,
etc.~\cite{rasmussen_2006}. The analogous operations for stable
distributions are highly nonlinear, often requiring technical
multivariable optimization and Monte Carlo codes, 
slowing down the resulting calculation many-fold. 
In this work, for these reasons, we 
restrict out attention to one-dimensional stable
distributions. Numerical schemes for multivariate stable distributions
are an area of current research.

To be more precise, if we denote by $\alpha \in (0,2]$ the {\em
  stability parameter}, $\beta \in [-1, 1]$ the {\em skewness
  parameter}, $\gamma \in \R$ the {\em location parameter}, and
$\lambda \in \R^+$ the {\em scale parameter of $X$}, then these random
variables satisfy the relationship~\cite{zolotarev_1986}:
\begin{equation}
a_1 X_1 + a_2 X_2 \sim a X + 
\begin{cases}
\lambda \gamma (a_1 + a_2 - a) & \alpha \neq 1\\
\lambda \beta(2 /\pi)(a_1 \log(a_1/a) + a_2 \log(a_2/a)) & \alpha =1  
\end{cases}
\qquad
\end{equation}
where, as before $\sim$ is used to denote {\em equality in
  distribution} 
and $a = (a_1^\alpha + a_2^\alpha)^{1/\alpha}$.
Enforcing the previous stability laws, although partially
redundant, places conditions on the characteristic function of $X$
(i.e. the Fourier transform of the probability density
function). Since the
density of the sum of two iid random variables is obtained via
convolution of their individual densities, in the Fourier domain this
is equivalent to multiplication of the characteristic functions.
It can be shown that, in general, 
the class of characteristic functions for stable
distributions must be of the form:
\begin{equation}\label{eq_char}
  \begin{aligned}
    \E \left[ e^{itX} \right] &= \varphi_X(t) \\ 
    &= 
e^{  \lambda ( it\gamma - \vert t \vert^\alpha + 
      it\omega(t,\alpha,\beta)) },
  \end{aligned}
\end{equation}
with
\begin{equation}
    \omega(t,\alpha,\beta) = 
    \begin{cases}
      \vert t \vert^{\alpha-1} \beta \tan\frac{\pi \alpha}{2} \quad & \text{if }
      \alpha \neq 1, \\
      -\frac{2 \beta}{\pi} \log|t| & \text{if } \alpha = 1,
    \end{cases}
\end{equation}
and as before,
\begin{equation}
  \alpha \in (0,2], \quad
  \beta \in [-1,1], \quad
  \gamma \in (-\infty,\infty), \quad
  \lambda > 0.
\end{equation}
This particular parameterization of the characteristic function
$\varphi_X$ in terms of $\alpha$, $\beta$, $\gamma$, and $\lambda$ is
the canonical one~\cite{zolotarev_1986}, and often referred to as the
$\vct{A}$-parameterization.  As discussed in Section~\ref{sec_stable},
we will deal solely with an alternative parameterization, the
$\vct{M}$-parameterization.  This parameterization is obtained by
merely a change of variables in the $x$-parameter but, in contrast
to~\eqref{eq_char}, is jointly continuous in all of its parameters.

Often, the form of the above characteristic function is taken to be
the {\em definition} of stable distributions because of the absence of
an analytic form of the inverse transform.
Special cases of these distributions are normal random variables
($\alpha=2$ and $\beta = 0$), Cauchy distributions ($\alpha=1$ and
$\beta=0$), and the L\'evy distribution ($\alpha=0.5$ and
$\beta=1$). Each of these distributions has a closed-form expression
for its density and cumulative distribution function. However, as
mentioned before, in
general, the density and distributions functions for stable random
variables have no known analytic form, and are expressible only via
their Fourier transform or special-case asymptotic series. Because of
this, performing inference or developing models based on these
distribution laws can be computationally intractable if the density
and distribution functions are expensive to compute (i.e. if
numerically evaluating the corresponding Fourier integral is expensive).
We will focus our attention on the numerical evaluation of the density
function for a unit, centered distribution: $\gamma = 0$ and $\lambda
= 1$. We will denote this class of unit, centered stable distributions
in the $\vct{A}$-parameterization 
as $\mathcal S(\alpha,\beta,\vct{A})$, and say that $X \sim \mathcal
S(\alpha,\beta,\vct{A})$ if $X$ has characteristic function~\eqref{eq_char}
with $\gamma = 0$ and $\lambda = 1$. In a slight change of notation
from~\cite{nolan_2015}, we
make the particular parameterization explicit in the definition of
$\mathcal S(\alpha, \beta, \cdot)$.

Most existing numerical methods for the evaluation of the
corresponding density function $f$,
\begin{equation}
f(x; \alpha, \beta) = \frac{1}{2\pi} \int_{-\infty}^\infty
\varphi_X(t) \, e^{-itx} \, dt,
\end{equation}
rely on some form of numerical integration~\cite{nolan_1997} or,
in the symmetric case ($\beta = 0$),
asymptotic expansions for extremal values 
of $x$, and $\alpha$~\cite{matsui_2004}.
Often, if the shape parameters of the
distribution are being inferred or estimated from data, as in the case
of maximum likelihood calculations or Bayesian modelling, the density
function~$f$ must be evaluate at the same $x$ values for {\em many}
values of the parameters~$\alpha$ and~$\beta$. In order to ensure the
accuracy of the numerical integration scheme, often adaptive quadrature
is used. However, for new values of the parameters, no previous
information can be used, nor are these quadratures optimal (as in the
sense that the trapezoidal rule is optimal for smooth periodic
function or that Gauss-Legendre quadrature is
optimal for polynomials on finite intervals~\cite{dahlquist_2003}).

The main contribution of this work is to develop an efficient means by
which to evaluate the density function for stable random variables
using various integral formulations, optimized quadrature schemes, and
asymptotic expansions. We develop generalized Gaussian quadrature
rules~\cite{bremer_2010b, yarvin_1998} which are able to evaluate the
density $f$ for various {\em ranges} of the shape parameters $\alpha$,
$\beta$, as well as the argument $x$. We generate quadratures that
consist of a {\em single} set of weights and nodes which are able to
integrate the characteristic function (or deformations thereof) for
large regions of the parameters and argument space.  Using a small
collection of such quadrature rules we are able to cover most of 
$\alpha \beta x$-space. This class of quadrature schemes is an extension of
the classical Gaussian quadrature schemes for polynomials which are
able to {\em exactly} integrate polynomials of degree $d \leq 2n - 1$
using $n$ nodes and $n$ weights (i.e. using a total of $2n$ degrees of
freedom).  We discuss these quadrature rules in detail in
Section~\ref{sec_gengauss}, as well as 
derive new asymptotic expansions in the asymmetric case in the 
$\vct{M}$-parameterization.
For regions with large~$x$, we derive efficient 
asymptotic expansions which can be used for evaluation.

The paper is organized as follows: in Section~\ref{sec_stable} we
review some standard integral representations and asymptotic
expansions for the density functions of stable distributions.  In
Section~\ref{sec_gengauss} we discuss the procedure for constructing
generalized Gaussian quadratures for evaluating the integral
representations presented in the previous section.  In
Section~\ref{sec_algorithm}, we demonstrate the effectiveness of our
schemes for evaluating the density functions via various numerical
examples.  In Section~\ref{sec_conclusions}, the conclusion, we
discuss some additional areas of research, and point out regimes in
which the algorithms of this paper are not applicable.

\section{Stable distributions}
\label{sec_stable}

In this section we review some basic facts regarding stable
distributions, and review the integral and asymptotic expansions that
we will use to evaluate the density function.
As mentioned in the previous section, there are several different
parameterizations of stable densities. We now detail the
parameterization useful for numerical calculations, most commonly
referred to as Zolotarev's \mbox{$\vct{M}$-parameterization}. Random 
variables that follow stable distributions with 
this parameterization will be
denoted $X \sim \mathcal S(\alpha, \beta, \vct{M})$.

\subsection{Basic Facts}
There are a number of different parameterizations for stable
distributions, each of which is useful in a particular regime:
integral representations, asymptotic expansions, etc. 
It was shown in \cite{nolan_1997} that Zolotarev's 
$\vct{M}$-parameterization is particularly useful for numerical
computations as it allows for the computation of a {\em unit} density
that can be later scaled and translated. 
In our numerical scheme, we also use the $\vct{M}$-parameterization because
of its continuity in all underlying parameters. This is a standard
procedure among other numerical methods as well.

This parameterization, and therefore the density function,
 is defined by the characteristic function
\begin{equation}\label{eq_charac}
\begin{aligned}
\varphi_X(t) &= 
e^{ \lambda ( it \gamma - |t|^\alpha + 
it \omega_M(t, \alpha, \beta) ) }, \\
\omega_M(t, \alpha, \beta) &= 
\begin{cases}
  (|t|^{\alpha-1} - 1) \beta \tan\frac{\pi \alpha}{2} \quad 
  & \text{if } \alpha \neq 1, \\
  - \frac{2\beta}{\pi} \log |t| & \text{if } \alpha = 1.
\end{cases}
\end{aligned}
\end{equation}
For the rest of the paper, we will work with unit stable laws 
($\gamma = 0$, $\lambda = 1$) unless otherwise mentioned.
We will refer to the case of $\beta=0$ as the {\em symmetric} case,
and otherwise for $\beta \neq 0$ the {\em asymmetric} case.
In general, the parameters $\alpha$ and $\beta$ cannot be
interchanged between different parameterizations, and in this
particular case the change of variables from the $\vct{A}$- to the
$\vct{M}$-parameterization is given by:
\begin{equation}
\alpha_A =  \alpha_M = \alpha,
\qquad \beta_A = \beta_M  = \beta,
\qquad \gamma_A = \gamma_M - \beta \tan \frac{\pi \alpha}{2} , 
\qquad \lambda_A = \lambda_M
\end{equation}
where the subscripts are used to denote the parameterization.
Under this change of variable,
the characteristic function in the $\vct{A}$-parameterization
lacks the term $-it \beta \tan (\pi \alpha /2)$ in the exponent.
The existence of this term in the $\vct{M}$-parameterization
makes characteristic function continuous at $\alpha = 1$.
This change of variables is mostly done for analytical convenience.
However, the mode of the density in the $\vct{A}$-parameterization
approaches infinity as $\alpha \to 1$ if $\beta \neq 0$.
Therefore, neither of the one-sided limits $\alpha \to 1^\pm $ 
equals a distribution in the $\vct{A}$-parameterization~\cite{zolotarev_1986}.

From~\eqref{eq_charac}, it follows that for unit stable laws,
\begin{equation}
\label{eq_char_symmetry}
\varphi_X(-t; \alpha, \beta) = \varphi_X(t; \alpha, -\beta) 
= \overline{ \varphi_X(t; \alpha, \beta) }, \\
\end{equation}
where $\overline{z}$ denotes the complex-conjugate of $z$, and 
we have used the Fourier transform convention
\begin{equation}
  \begin{aligned}
    \varphi_X(t) &= \E \left[ e^{itX} \right] \\
    &= \int_{-\infty}^\infty f(x) \, e^{itx} \, dx.
  \end{aligned}
\end{equation}
The density $f$ is therefore given by:
\begin{equation}
f(x) = \frac{1}{2\pi} \int_{-\infty}^\infty \varphi_X(t) \, e^{-itx}
\, dt.
\end{equation}
In conjunction with \eqref{eq_char_symmetry}, one can show that
\begin{equation}
\label{eq_inverse_Fourier_abstract}
\begin{aligned}
  f(x;\alpha, \beta) & = \frac{1}{2\pi} \int_{-\infty}^\infty e^{-itx}
  \, \varphi_X(t; \alpha, \beta) \, dt \\
  & = \frac{1}{2\pi} \left( \overline{ \int_0^\infty e^{itx} \, 
      \varphi_X(t; \alpha, -\beta) \, dt } 
    + \int_{-\infty}^{0} e^{-itx} \, \varphi_X(t; \alpha, \beta) \right) \\
  & = \frac{1}{\pi} \Re \int_0^\infty e^{itx} \,
  \varphi_X(t;\alpha,-\beta)   \, dt. \\
\end{aligned}
\end{equation}
Furthermore,~\eqref{eq_char_symmetry} and the Fourier inversion
formula imply that
\begin{equation}
  \label{eq_den_symmetry}
  f(-x; \alpha, \beta) = f(x; \alpha, -\beta).
\end{equation}
This symmetry allows for considerable restriction in the relevant
values of $x$. In particular, for $\alpha \neq 1$, defining
$\zeta(\alpha,\beta) = \beta\tan\pi\alpha/2$, we need only address
the case $x > \zeta$. 
Indeed, if $x < \zeta$, then 
\begin{equation}
-x > -\zeta(\alpha, \beta) = \zeta(\alpha, -\beta)
\end{equation}
as can be seen from~\eqref{eq_charac}.

\subsection{Integral representations}
\label{sec_intrep}

Inserting $\varphi_X$ from~\eqref{eq_charac} into the inverse Fourier
transform~\eqref{eq_inverse_Fourier_abstract}, we see that 
\begin{equation}
\label{eq_inverse_Fourier}
  f(x; \alpha, \beta) = \frac{1}{\pi} \int_0^\infty
  \cos( h(t;x, \alpha, \beta)) \, e^{-t^\alpha} \, dt,
\end{equation}
where for $\alpha \neq 1$
\begin{equation}
  \begin{aligned}
    h(t; x, \alpha, \beta) &= (x - \zeta) t + \zeta t^\alpha, \\
    \zeta(\alpha,\beta)  &= -\beta \tan \frac{\pi\alpha}{2},
  \end{aligned}
\end{equation}
and for $\alpha=1$
\begin{equation}
  \begin{aligned}
    h(t; x, \alpha, \beta) &=    xt + \frac{2\beta t}{\pi}  \log t, \\
    \zeta(\alpha,\beta)  &= 0.
  \end{aligned}
\end{equation}
Because of the linear dependence of $h$ on $t$, the integrand
in~\eqref{eq_inverse_Fourier} is oscillatory for modestly sized values
of $x$.  See Figure~\ref{fig_oscillatory} for a plot of this
integrand.  For numerical calculations, the infinite interval of
integration can be truncated based on the decay of the exponential
term if $x$ is not too large.  However, for small values of $\alpha$,
this region of integration can still be prohibitively large.
Furthermore, for large $\vert x\vert$, the integrand becomes
increasingly oscillatory and standard integration schemes
(e.g. trapezoidal rule, Gaussian quadrature) not only become
expensive, but lose accuracy due to the oscillation.  It is possible
that Filon-type quadratures~\cite{olver_2008} could be applicable, but
this has yet to be thoroughly investigated in the literature.
Section~\ref{sec_altggquad} contains a brief discussion of quadrature
techniques for highly oscillatory integrands, but these methods,
unfortunately, would likely prove to be more computationally expensive
than those techniques presented in this work.
 For these reasons, this
 representation of the density $f$ in~\eqref{eq_inverse_Fourier}
 cannot be used efficiently in the following 
parameter ranges: small $\alpha$ and/or large $\vert x \vert$.

\begin{figure}[t!]
  \begin{center}
    \includegraphics[width=.5\linewidth]{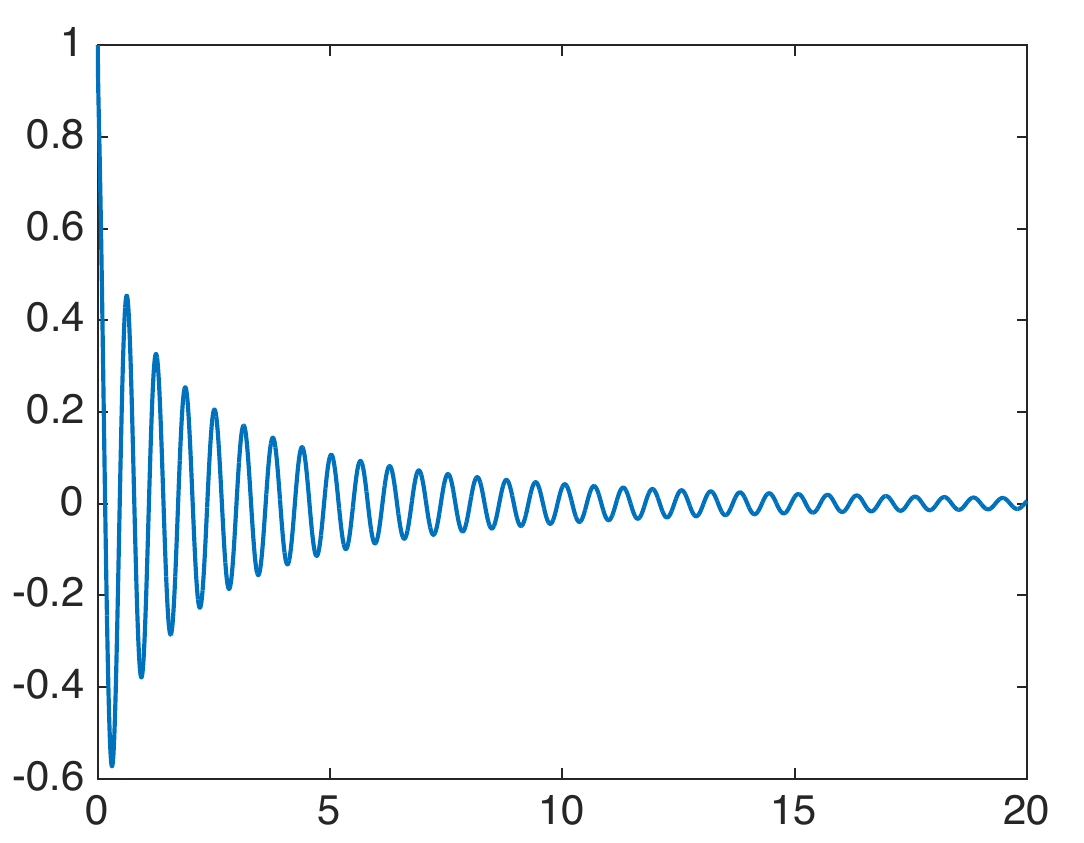} 
  \end{center}
  \caption{Graph of the Fourier transform of $f$, 
    i.e. the integrand in $\eqref{eq_inverse_Fourier}$, for parameter
    values of $x = 10$, $\alpha = .5$ and $\beta = 0$.}
  \label{fig_oscillatory}
\end{figure}

Alternatively, the integral in~\eqref{eq_inverse_Fourier} can be
rewritten using the method of stationary phase. This calculation was
done in~\cite{nolan_1997}. To this end, we begin by
rewriting~\eqref{eq_inverse_Fourier} as
\begin{equation}
f(x; \alpha, \beta) = \frac{1}{\pi} \Re \int_0^\infty e^{ih(z) -
  z^\alpha} \, dz,
\end{equation}
where $\Re  w$ denotes the real part of the complex number $w$ and we
have suppressed the explicit dependence on $\alpha$ and $\beta$ for
simplicity. Deforming along the contour with zero phase, we have
\begin{equation}\label{eq_statphas}
f(x; \alpha, \beta) = \frac{\alpha}{\pi |\alpha - 1| } 
\frac{1}{(x - \zeta)} \int_{-\theta_0}^{\frac{\pi}{2}}
g(\theta;x,\alpha,\beta) \, 
e^{-g(\theta;x,\alpha,\beta) } \, d\theta,
\end{equation}
with 
\begin{equation}
g(\theta;x,\alpha,\beta) 
= \left( x - \zeta \right)^{\frac{\alpha}{\alpha-1}} \, V(\theta;\alpha,\beta),
\end{equation}
where for $\alpha \neq 1$
\begin{equation}
  \begin{aligned}
    \zeta(\alpha,\beta) &= -\beta \tan \frac{\pi\alpha}{2}, \\
    \theta_0(\alpha,\beta) &= \frac{1}{\alpha} 
    \arctan\left(\beta \tan \frac{\pi \alpha}{2} \right), \\
    V(\theta;\alpha,\beta) &=
    \left( \cos \alpha\theta_0 \right)^{ \frac{1}{\alpha - 1} } \left(
      \frac{\cos \theta}{\sin \alpha(\theta_0 + \theta) } \right)^{
      \frac{\alpha}{\alpha-1}} \frac{ \cos \left( \alpha \theta_0 +
        (\alpha-1) \theta \right) }{ \cos \theta},
\end{aligned}
\end{equation}
and where for $\alpha = 1$ 
\begin{equation}
  \begin{aligned}
    \zeta(\alpha,\beta) &= 0, \\
    \theta_0(\alpha,\beta) &= \frac{\pi}{2}, \\
    V(\theta;\alpha,\beta) &= \frac{2}{\pi} \left( \frac{
        \frac{\pi}{2} + \beta \theta }{\cos \theta} \right) \exp
    \left( \frac{1}{\beta} \left( \frac{\pi}{2} + \beta\theta \right)
      \tan \theta \right).
  \end{aligned}
\end{equation}
While seemingly more complicated than that
in~\eqref{eq_inverse_Fourier}, the integrand in~\eqref{eq_statphas} is
strictly positive, has no oscillations and the interval of integration
is finite.  Unfortunately, this is not a fail-safe transformation.
In particular, for very small and very large $x$, and $\alpha$ close
to $1$ and $2$, the integrand has large derivatives (e.g. very spiked)
and is hard to efficiently integrate, see Figure~\ref{fig_spiked}.
 For this reason, previous
schemes~\cite{matsui_2004, nolan_1997} have used zero-finding methods
to locate the integrand's unique extremum point $\theta_{max}$,
where~$g(\theta_{max}) = 1$.  Subsequently, adaptive quadrature
schemes were applied to the two subintervals created by splitting the
original interval of integration at~$\theta_{max}$.
This procedure is often computationally expensive.

\begin{figure}[t]
  \begin{center}
    \includegraphics[width=.5\linewidth]{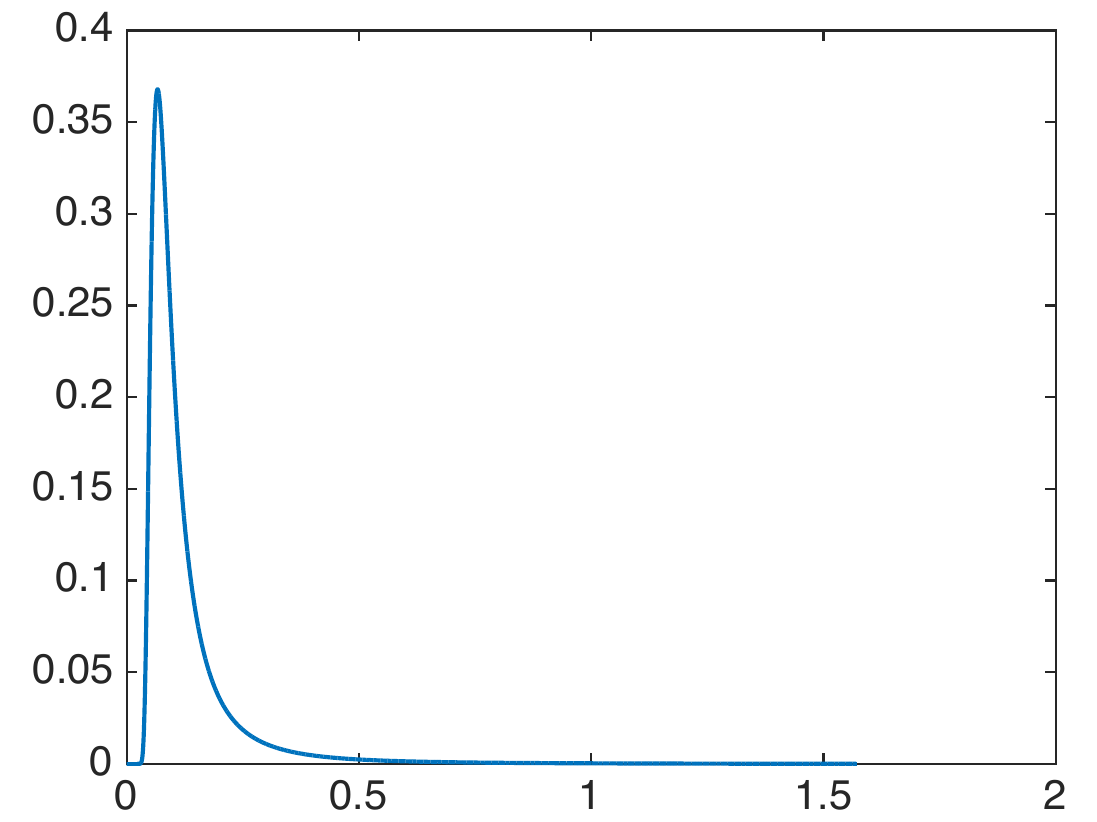} 
  \end{center}
  \caption{Graph of the stationary phase integrand in $\eqref{eq_statphas}$ for
    parameter values of $x = 10^{-1}$, $\alpha = 1.5$ and $\beta =
    0$. Note the large derivatives and spiked behavior.}
  \label{fig_spiked}
\end{figure}

Even though an expression for the cumulative distribution
function (CDF) can be derived by straightforward integration
of~\eqref{eq_statphas}, this is not ideal for various numerical
considerations described later.
We obtain an expression for $F$, the CDF for $f$, by 
using an inversion theorem found in~\cite{shep_1991}:
\begin{equation}    \label{eq_cdf}
  \begin{aligned}
    F(x) &= \frac{1}{2} - \frac{1}{2\pi} \int_0^\infty
    \left(\varphi(t)e^{-ixt}  - \varphi(-t) \, e^{ixt} \right) \, 
    \frac{dt}{it} \\
    &= \frac{1}{2} + \frac{1}{\pi} \int_0^\infty \sin(h(t; x, \alpha,
    \beta)) \, e^{-t^\alpha} \, \frac{dt}{t},
\end{aligned}
\end{equation}
The theorem in \cite{shep_1991} assumes
the existence of a mean of the random variable which is 
associated with the characteristic function.
However, stable variates with $\alpha < 1$ do not have a mean.
Fortunately, one can relax the assumptions of the theorem
to integrability of the integrand in \eqref{eq_cdf}.
This means that the expression is valid for all parameter combinations.
The integrand of~\eqref{eq_cdf} has similar behavior as 
the integrand in~\eqref{eq_inverse_Fourier}, and therefore we expect
that similar numerical schemes for evaluating the integral will be
applicable.
In the asymmetric case however, the integrand has a singularity at the origin
when $\alpha < 1$ and $\beta \neq 1$.
The advantages of this representation are explained in detail in
Section~\ref{sec_algorithm}.

\subsection{Series and asymptotics}
\label{sec_asymp}

Fortunately, there are series and asymptotic expansions for
Zolotarev's $\vct{M}$-parameterization which nicely compliment 
the integral representations above.  Specifically, they yield accurate
results for very small and large $x$.  Zolotarev derived series
and asymptotic expansions for the $\vct B$-parameterization 
in~\cite{zolotarev_1986}, and in the following, we will 
derive similar representations for the
$\vct M$-parameterization valid in the general case\footnote{After the publication of this article, we were made aware of related prior work on asymptotic expansions in \cite{teimouri2008stable}.}.

\begin{lemma}
\label{lem_series_0}
Let $\alpha \neq 1$, $\beta \in [-1,1]$, $\zeta = -\beta \tan( \pi \alpha/2)$, and 
\begin{equation}
  \label{eq_series_at_zero}
  S^0_{n}(x; \alpha, \beta ) := \frac{1}{\alpha \pi} \sum_{k=0}^{n} 
  \frac{ \Gamma( \frac{k+1}{\alpha} )}{ \Gamma( k+1) } 
  (1+\zeta^2)^{ - \frac{k+1}{ 2\alpha}} 
  \sin \left( \left[ \pi/2 + (\arctan \zeta)/\alpha \right]
    (k+1) \right) (x - \zeta)^k.
\end{equation}
Then for any $n \in \N$, 
\begin{equation}
\label{eq_error_estimate_at_0}
|f(x; \alpha, \beta ) - S_{n-1}^0(x; \alpha, \beta)| \leq \frac{1}{\alpha \pi} 
		\frac{ \Gamma( \frac{n+1}{\alpha} )}{ \Gamma( n+1) } 
		(1+\zeta^2)^{-\frac{n+1}{2\alpha}} 
		| x - \zeta |^n.
\end{equation}
\end{lemma}

\begin{proof}
To obtain a series representation centered at $x = \zeta$, 
we follow the derivation in~\cite{zolotarev_1986}, but
for the $\vct{M}$-parameterization instead of the $\vct{B}$-parameterization:
\begin{equation}
\label{eq_0der1}
\begin{aligned}
  f(x;\alpha, \beta) & = \frac{1}{\pi} \Re \int_0^\infty e^{itx} \, 
  \varphi_X(t;\alpha,-\beta) \, dt \\
  & = \frac{1}{\pi} \Re \int_0^{\infty} \sum_{k=0}^\infty  
  \frac{ (it \left[ x-\zeta\right] )^k }{k!} 
  \exp\left( - (1-i\zeta) t^\alpha \right) \,  dt \\
  & = \frac{1}{\pi} \sum_{k=0}^{n-1}  
  \frac{( x-\zeta )^k }{k!}  \Re \int_0^{\infty} (it)^k
  \exp\left( - (1-i\zeta) t^\alpha \right) \,  dt  + R_n,
\end{aligned}
\end{equation}
where
\begin{equation}
R_n = \frac{1}{\pi} \Re \int_0^\infty \left[ e^{it (x- \zeta) } - \sum_{k=0}^{n-1} \frac{ (it [x-\zeta] )^k }{k!}\right]
		\exp( - (1-i \zeta ) t^\alpha ) \ dt
\end{equation}

Applying the change
of variables \mbox{$s = (1-i\zeta)^{1/\alpha} t$} to the last integral 
in \eqref{eq_0der1} and subsequently
rotating the contour of integration to the real axis yields
\begin{equation}
  \begin{aligned}
    f(x;\alpha, \beta) &= S_{n-1}(x;\alpha, \beta) + R_n^0.
  \end{aligned}
\end{equation}
The change of contour can be justified with Lemma 2.2.2
in~\cite{zolotarev_1986}.  It remains to show that $R_n$ is bounded in
magnitude by~\eqref{eq_error_estimate_at_0}.  Indeed,
\begin{equation}
\begin{aligned}
|R_n | & = \frac{1}{\pi}  \left| \Re \int_0^\infty \left( e^{it(x-\zeta)} - \sum_{k=0}^{n-1} \frac{[it(x-\zeta)]^k}{k!} \right) e^{ - (1-i\zeta) t^\alpha }  \ dt \right| \\
& = \frac{1}{\pi} \left| \Re \int_0^\infty \left( \sum_{k=n}^\infty \frac{( i [1-i\zeta]^{-1/\alpha} s [x-\zeta])^k}{k!} \right) e^{ - s^\alpha } \ \frac{ds}{(1-i\zeta)^{1/\alpha}} \right| \\
& \leq \frac{ |x-\zeta|^n}{\pi n!} (1+\zeta^2)^{-\frac{n+1}{2 \alpha}} \int_0^\infty s^n e^{-s^\alpha} ds \\
& = \frac{1}{\alpha \pi} 
		\frac{ \Gamma( \frac{n+1}{\alpha} )}{ \Gamma( n+1) } 
		(1+\zeta^2)^{-\frac{n+1}{2\alpha}} 
		| x - \zeta |^n.
\end{aligned}
\end{equation}
The second equality comes from the change of variables
\mbox{$s = (1-i \zeta)^{1/\alpha} t$}, and a rotation of the contour
of integration to the real axis.  The difference between the
exponential and the sum of the first $n-1$ terms of its power series
can bounded by the $n^{th}$ term with the mean value theorem.
\end{proof}

As a consequence, the series~\eqref{eq_series_at_zero} converges
to the density as $n \to \infty$ for $\alpha > 1$.
For $\alpha < 1$, \eqref{eq_series_at_zero} is not convergent, 
but can be used as an
asymptotic expansion as $x \to \zeta$ if $\beta \neq 1$.
While the truncation error bound \eqref{eq_error_estimate_at_0}
holds regardless of the parameters, \eqref{eq_series_at_zero}
does not capture the asymptotic behavior of the
density as $x \to \zeta^+$ if $\alpha < 1$ and $\beta = 1$.
Indeed, \eqref{eq_series_at_zero} is identically zero for this parameter choice.
On the other hand, the density falls off exponentially as $x \to \zeta^+$,
which an asymptotic expansion in~\cite{zolotarev_1986} reveals.
Unfortunately, this expansion cannot be efficiently evaluated numerically
because its coefficients do not have a closed form.

Still, by rearranging \eqref{eq_error_estimate_at_0}, we find that
truncating~\eqref{eq_series_at_zero} after the $n-1$ term
is accurate to within~$\eps$ of the true value for all~$x$
satisfying
\begin{equation}
\begin{aligned}
  |x-\zeta| &\leq  \left[ \eps \alpha \pi 
    (1+\zeta^2)^{\frac{n+1}{2\alpha} } 
    \frac{ \Gamma(n+1)}{\Gamma(\frac{n+1}{\alpha}) } \right]^{1/n} \\
&:= B^0_n(\alpha, \beta).
\end{aligned}
\end{equation}
After the discussion in the paragraph above, it is important to stress that
this bound guarantees {\it absolute accuracy} of the truncated series
, rather than {\it relative accuracy}.

\begin{lemma}
Let $\alpha \neq 1$, $\beta \in [-1,1]$, $\zeta = -\beta \tan( \pi \alpha/2) $, and 
\begin{equation}
  \label{eq_series_at_infinity}
  S^\infty_{n}(x; \alpha, \beta ) := \frac{\alpha}{\pi} \sum_{k=1}^n (-1)^{k+1} 
    \, \frac{\Gamma(\alpha k)}{\Gamma(k)} \, (1+\zeta^2)^{k/2}  \, 
    \sin([\pi \alpha / 2 - \arctan \zeta] k) \, (x-\zeta)^{-\alpha k-1} .
\end{equation}
Then for any $n \in \N$, 
\begin{equation}
  \label{eq_error_estimate_at_infinity}
  |f - S_{n-1}^\infty| \leq \frac{\alpha}{ \pi} 
  \frac{ \Gamma( \alpha n)}{ \Gamma( n ) } 
  (1+\zeta^2)^{\frac{n}{2}} 
  | x - \zeta |^{-\alpha n - 1} .
\end{equation}
\end{lemma}

\begin{proof}
For simplicity, we first derive a series expansion for the $\vct{A}$-parameterization 
and convert it to the $\vct{M}$-parameterization via
the shift $x_A = x_M - \zeta$ afterwards.
To do this, we extend $\varphi_X$ to the complex plane, and
integrate along the contour $z = iu x^{-1/\alpha}$.
Again, this is justified by Lemma 2.2.2 in~\cite{zolotarev_1986}.
\begin{equation}
\label{eq_infder1}
\begin{aligned}
  x^{-1/\alpha} f_A(x^{-1/\alpha}; \alpha, \beta) & = \frac{1}{\pi x^{1/\alpha} } \Re \int_0^\infty e^{izx^{-1/\alpha} } \, 
  \varphi_X(z;\alpha,-\beta) \, dz \\
  	& = - \frac{1}{\pi} \Im \int_0^\infty e^{-u} \varphi(iux^{1/\alpha}; \alpha, -\beta) \, du \\
	& = \frac{\alpha}{\pi} \sum_{k=1}^{n-1} \frac{\Gamma(\alpha n)}{ \Gamma(n)} 
			(-1)^{k+1} (1+\zeta^2)^{k/2} \sin( [ \pi \alpha/2 - \arctan \zeta] k ) \, x^k + R^\infty_n.
\end{aligned}
\end{equation}
Here, the first $n$ terms of the power series of the characteristic function
were used to approximate the integral.
Therefore,
\begin{equation}
R^\infty_n := - \frac{1}{\pi} \Im  \int_0^\infty \left[ \exp( - (1-i\zeta) x u^\alpha) - 
							\sum_{k=0}^{n-1} \frac{ ( - [1-i\zeta] x u^\alpha )^k }{k!} \right] 
							e^{-u}\, du.
\end{equation}
The bound 
of $|R^\infty_n|$
is attained in a similar manner as the one for $R^0_n$ in Lemma \ref{lem_series_0}.
Rearranging the last line of \eqref{eq_infder1}, and substituting 
$x- \zeta$ for $x$ yields the series \eqref{eq_series_at_infinity}
 in the \mbox{$\vct{M}$-parameterization}.
\end{proof}

Expression~\eqref{eq_series_at_infinity}
converges to the density for $\alpha < 1$, 
and can be used as an asymptotic expansion for $\alpha > 1$, $\beta \neq -1$.
In the case $\alpha > 1$, $\beta = -1$, \eqref{eq_series_at_infinity}
is identically zero, while the density decreases to 
zero exponentially as $x \to \infty$~\cite{zolotarev_1986}.
As with the series above however, we can still guarantee absolute accuracy 
compared to the true density.
Indeed, as a consequence of the lemma, the series \eqref{eq_series_at_infinity} is accurate to precision~$\eps$ for any $x$ satisfying 
\begin{equation}
\begin{aligned}
  |x-\zeta| &\geq  \left[ \frac{\alpha}{\pi \eps} 
    (1+\zeta^2)^{\frac{n}{2} }
    \frac{ \Gamma(\alpha n)}{\Gamma( n) } 
  \right]^{1/(\alpha n-1)} \\
&:= B_{n-1}^\infty(\alpha, \beta).
\end{aligned}
\end{equation}
Notably, if for some $\alpha, \beta$, we take $n_0(\alpha, \beta)$ terms 
of \eqref{eq_series_at_zero} and $n_\infty(\alpha, \beta)$ terms
of \eqref{eq_series_at_infinity}, then it only remains to show
that values of $x$ in the range
\begin{equation}
B_{n_0}^0 \leq x - \zeta \leq B_{n_\infty}^\infty
\end{equation}
can be evaluated efficiently.
We will elaborate on the details of our scheme
in Section~\ref{sec_algorithm}.

\subsection{Derivatives of stable densities}
\label{sec_derivatives}

The integral representation \eqref{eq_inverse_Fourier} admits
reasonably simple expressions for the Fourier transform of the
derivatives of the density with respect to $x$, $\alpha$, and $\beta$.
We will derive these expressions here.  For brevity, let
$h= h(t, x;\alpha,\beta)$, $\zeta = \zeta(\alpha, \beta)$, and
$\partial_x = \partial/\partial x$. Similarly for differentiation with
respect to $\alpha$ and $\beta$.
First, note that for $\alpha \neq 1$,
\begin{equation}
\begin{aligned}
\partial_\alpha \zeta &= - \frac{\pi}{2} \beta \left[ \left( \tan \frac{\pi \alpha }{2} \right)^2 + 1 \right], \\
\partial_\beta \zeta &= - \tan \frac{\pi \alpha}{2},
\end{aligned}
\end{equation}
and
\begin{equation}
\begin{aligned}
\label{eq_hder}
\partial_x h &= t, \\
\partial_\alpha h &= (t^\alpha - t) \partial_\alpha \zeta + t^\alpha \log(t) \zeta, \\
\partial_\beta h &= (t^\alpha - t) \partial_\beta \zeta.
\end{aligned}
\end{equation}
  In order to obtain expressions for $\partial_\alpha h$ and $\partial_\beta h$ 
  at $\alpha = 1$, we compute the limit as $\alpha \to 1$ of the 
  corresponding expressions in~\eqref{eq_hder}:
\begin{equation}
\begin{aligned}
\label{eq_hder_a1}
\lim_{\alpha\to 1}\partial_\alpha h &= \frac{\beta}{\pi} t \log^2 t, \\
\lim_{\alpha\to 1}\partial_\beta h &= \frac{2}{\pi} t \log t .
\end{aligned}
\end{equation}
Since $h$ is continuous in all parameters at $\alpha = 1$
and both one-sided limits exist, the values of $\partial_\alpha f$ 
and $\partial_\beta f$ are well-defined when $\alpha=1$.
Finally, we have
\begin{equation}
\begin{aligned}
\label{eq_partials}
\partial_x f(x; \alpha, \beta) &= -\frac{1}{\pi} \int_0^\infty t \, \sin h
\, e^{-t^\alpha} \, dt, \\
\partial_\alpha f(x; \alpha, \beta) &= -\frac{1}{\pi} \int_0^\infty 
\left(  \sin h \, \partial_\alpha h +  t^\alpha \cos h \,  \log t \right) 
e^{-t^\alpha} \, dt,  \\
\partial_\beta f(x; \alpha, \beta) &= -\frac{1}{\pi} \int_0^\infty
\sin h \, \partial_\beta h \, e^{-t^\alpha} \, dt.
\end{aligned}
\end{equation}
The partial derivatives in~\eqref{eq_partials}
have a relatively compact form.  In contrast,
the partial derivatives with respect to~$\alpha$ and~$\beta$ of the
stationary phase integral~\eqref{eq_statphas} and the series
expansions~\eqref{eq_series_at_zero} and~\eqref{eq_series_at_infinity}
become rather unwieldy.  Nevertheless, $\partial_\alpha$ of the
stationary phase integral and series representation of $f$ was
computed in~\cite{matsui_2004} (for the symmetric case).
However,
this approach becomes cumbersome in the general case, as numerous
applications of the product and chain rule make the expressions
impractically long.

In order to compute the derivatives of the series
representations~\eqref{eq_series_at_zero} 
and~\eqref{eq_series_at_infinity} with
respect to $x$, differentiation can be done term-by-term. See
Appendix~\ref{app_grad} for this calculation.

\section{Generalized Gaussian quadrature}
\label{sec_gengauss}

In this section we briefly discuss what are known as {\em generalized 
 Gaussian quadrature rules}. These integration schemes are analogous
to the Gaussian quadrature rules for orthogonal
polynomials, except that they are applicable to wide classes of
functions, not merely polynomials. See~\cite{dahlquist_2003} 
for a description of
classical Gaussian quadrature with regard to polynomial integration.
Generalized Gaussian quadrature
schemes were first rigorously introduced
in~\cite{ma_1996,yarvin_1998}. Recently, a more efficient scheme for their
construction was developed
in~\cite{bremer_2010b}. It is this more recent algorithm that we base
our calculations on, and outline the main ideas here. See
both of these references for a detailed description of these
quadrature rules.

\subsection{Gaussian quadrature}

A $k$-point quadrature rule consists of a set of $k$ nodes and 
weights, which we will denote by $\{x_j,w_j\}$. These nodes and
weights are chosen to accurately approximate the integral of a
function $f$ with respect to a positive weight function $\omega$:
\begin{equation}
\int_a^b f(x) \, \omega(x) \, dx \approx \sum_{j=1}^k w_j \, f(x_j).
\end{equation}
Many different types of quadrature rules exist which exhibit different
behaviors for different classes of functions $f$. 
In short, if a $k$-point quadrature rule exists which integrates $k$
linearly independent functions $f_1$, \ldots, $f_{k}$ we say that
the quadrature rule is a \emph{Chebyshev quadrature}. If the $k$-point
rule is able to integrate $2k$ functions $f_1$, \ldots, $f_{2k}$ then
we say that the rule is \emph{Gaussian}.

In the case where the $f_\ell$ are polynomials, the nodes and weights
of the associate Gaussian quadrature can be determined from the
class of orthogonal polynomials with corresponding weight 
function~$\omega$. However, in the case where the $f_\ell$'s are arbitrary
square-integrable 
functions, these nodes and weights must be determined in a purely
numerical manner.

\subsection{Nodes and weight by nonlinear optimization}

We now provide an overview of the numerical 
procedure for constructing a Gaussian quadrature rule for the
integrand in equation~\eqref{eq_inverse_Fourier} using the procedure
of~\cite{bremer_2010b}.
Recall, we are constructing a quadrature rule to compute:
\begin{equation}\label{eq_ggquad}
f(x;\alpha,\beta) = \frac{1}{\pi} \int_0^\infty  \cos(
h(t;x,\alpha,\beta) ) \, e^{-t^\alpha} \, dt,
\end{equation}
i.e., the goal is to compute integrals of the functions we will
denote by
\begin{equation}
\phi(t;x,\alpha,\beta) = \cos(h(t;x,\alpha,\beta)) \, e^{-t^\alpha}.
\end{equation}
For reasons of clarity,
we address computing a generalized Gaussian quadrature
scheme for a
class of functions $\psi = \psi(t;\eta)$, i.e. those that depend on
only one parameter, $\eta$. The multi-parameter case is analogous. The
following discussion is cursory, and we direct the reader
to~\cite{bremer_2010b} for more details, as there are several aspects
of numerical analysis, optimization, and linear algebra that would
merely distract from the current application.

For a selection of $2n$ linearly-independent functions $\psi_\ell$, we note
that the corresponding $n$-point 
generalized Gaussian quadrature $\{t_j,w_j\}$ is the solution to the
following system of $2n$ non-linear equations:
\begin{equation}
  \begin{aligned}
    \sum_{j=1}^n w_j \, \psi_1(t_j) &= \int \psi_1(t) \, dt, \\
    \vdots \qquad &= \qquad \vdots \\
    \sum_{j=1}^n w_j \, \psi_{2n}(t_j) &= \int \psi_{2n}(t) \, dt.
  \end{aligned}
\end{equation}
Obtaining a solution to this system is the goal of the following
procedure.

The scheme proceeds by first 
finding an orthonormal set of functions $u_\ell$ such that
any $\psi(\cdot,\eta)$ can be approximated, to some specified
precision~$\epsilon$, as a 
linear combination of the~$u_\ell$ for any~$\eta$.
Next, an oversampled quadrature scheme that
integrates \emph{products of these functions} is constructed, by
using, for example, adaptive Gaussian
quadrature~\cite{press_2007}. Adaptive Gaussian quadrature proceeds by
dividing the interval of integration into several segments such that
on each segment, the integral is computed to a specified accuracy. The
accuracy of each quadrature is determined by comparing with the value
obtained on a finer subdivision of the interval.

For $2n$ functions, this means that we have a $m$-point quadrature rule
$\{t_j,w_j \}$ such that
\begin{equation}
\left| \int u_k(t) \, u_\ell(t) \, dt - \sum_{j=1}^m w_j \, 
u_k(t_j) \, u_\ell(t_j) \right| \leq \epsilon,
\end{equation}
for all $1\leq k,\ell\leq 2n$, with $m\geq 2n$.
Accurately integrating products of the $u_\ell$'s allows for stable
interpolation to be done for any $\eta \neq \eta_\ell$~\cite{bremer_2010b}.
At this point, the vectors $\vct{u}_\ell \in \mathbb R^m$ serve as finite
dimensional embeddings of the square-integrable functions $u_\ell$:
\begin{equation}
  \vct{u}_\ell = 
  \begin{bmatrix}
    \sqrt{w_1} \, u_\ell(t_1) \\
    \vdots \\
    \sqrt{w_m} \, u_\ell(t_m)
  \end{bmatrix}.
\end{equation}
Here, the $u_\ell(t_j)$'s are scaled so that
$\vct{u}_\ell^T \vct{u}_\ell \approx ||u_\ell||^2_2$.
Computing a rank-revealing $\vct{Q}\vct{R}$ decomposition of the
matrix $\vct{U}$,
\begin{equation}\label{eq_umat}
  \mtx{U} = 
  \begin{bmatrix}
    \vct{u}^T_1 \\
    \vdots \\
    \vct{u}^T_m
  \end{bmatrix},
\end{equation}
allows for the immediate construction of a $2n$-point Chebyshev 
quadrature rule. This procedure, equivalently, has selected $2n$ values
of $t_j$ that can serve as integration (and interpolation) nodes for
all of the $u_\ell$'s. Refining this $2n$-point Chebyshev quadrature
down to an $n$-point quadrature proceeds via a Gauss-Newton
optimization. On each step, a single node-weight pair $(t_i,w_i)$ is
chosen to be discarded and the remaining nodes and weights are
optimized.
The procedure proceeds until roughly $n$ nodes remain, or accuracy in
the resulting quadrature starts to suffer.
While the weights we obtained as a result of this optimization
procedure happened to be positive, no explicit effort was made to
ensure this. Theoretical considerations for the existence of positive
weights can be found in~\cite{yarvin_1998}.

\begin{remark}
  Note that in the symmetric case, we must obtain a selection of
  $x_{\ell_1}$, $\alpha_{\ell_2}$ that yield a (possibly redundant)
  basis for all
  $\phi$. This could be done via adaptive discretization in these
  variables, but in practice, we merely sample $x$, $\alpha$ at
  Chebyshev points in Region~I in Figure~\ref{fig_symmetric}. The
  parameter $\alpha$ is sampled at roughly $100$ Chebyshev points in
  $[0.5,2.0]$, and then for each of these values $\alpha_{\ell_2}$,
  $x$ is sampled at roughly $100$ Chebyshev point in the interval
  $[0,B^\infty_{40}(\alpha_{\ell_2})]$. This yields an initial set of
  10,000 functions which are then compressed and integrated.  In order
  to make sure this sampling in $x$ and $\alpha$ provides a suitable
  set of functions to span the space of all $\phi$, we rigorously test
  the quadrature at many thousands of random locations in Region~I,
  comparing against adaptive quadrature.  The asymmetric case is
  analogous, with equispaced sampling in $x$, $\alpha$, and
  $\beta$. Experimentally, in order to obtain high accuracy in the
  resulting quadrature, more nodes are required in the sampling of $x$
  than in $\alpha$ or $\beta$.
\end{remark}





\subsection{A rank-reducing transformation}
\label{sec_rank}

In the brief description of the above algorithm, we assumed that the
interval of integration for the quadrature was finite. In our case,
the interval in~\eqref{eq_ggquad} is infinite, but can be truncated
given that it decays quickly due to the term $e^{-t^\alpha}$.  A
common interval of integration for all $\phi(\cdot; x,\alpha)$ can be
obtained based on the decay of $e^{-t^\alpha}$ for the smallest
$\alpha$ under consideration. In fact, for a particular
precision~$\epsilon$, we can set the upper limit of integration to
be~$T_\alpha = (-\log\epsilon)^{1/\alpha}$. Using this limit, we can
redefine each integrand $\phi$ under a linear transformation:
\begin{equation}\label{eq_changevar}
  \begin{aligned}
    f(x;\alpha,\beta) &\approx \frac{1}{\pi} \int_0^{T_\alpha} 
    \phi(t;\alpha,\beta) \, dt \\
    &= \frac{T_\alpha}{\pi} \int_0^1 \phi(\tau T_\alpha;\alpha,\beta) \,d\tau
    \\
    &= \frac{T_\alpha}{\pi} 
    \int_0^1 \cos(h(\tau T_\alpha ; \alpha, \beta)) \, 
    e^{-(\tau T_\alpha)^\alpha} \, d\tau\\
    &= \frac{T_\alpha}{\pi} 
    \int_0^1 \tilde\phi(\tau ; \alpha, \beta) \,d\tau.
  \end{aligned}
\end{equation}
Computing generalized Gaussian quadratures for the functions
$\tilde\phi$ turns out to be much more efficient due to the
similarity of numerical support (i.e. those $t$ such that
$|\tilde\phi(t)| > \epsilon$). This change of variables can
significantly  reduce the rank obtained in  the rank-revealing
$\mtx{QR}$ step of the previous non-linear optimization procedure.
For example, the generalized Gaussian quadrature for functions in
Region~I in Figure~\ref{fig_symmetric} consisted of 100
nodes/weights
before the change of variables, and only 43 nodes/weights afterward.
The resulting quadrature can be applied to the original function
$\phi$ via a straightforward linear transformation of the nodes and
scaling of the weights.

To be fair, the stationary phase integral~\eqref{eq_statphas} too,
permits such a rank-reducing transformation. However, it turns out to
be much less efficient because it relies on an \emph{a priori}
zero-finding procedure. Notably,~\eqref{eq_statphas} has changing
numerical support for different choices of the input parameters $x$,
$\alpha$, and $\beta$.  This is true even after accounting for the
parameter-dependent interval of integration by composing the
integrand with a linear map from $[0,1]$ to $[-\theta_0, \pi/2]$.
Experimentally, even though the integrand decays to zero at
both $-\theta_0$ and $\pi/2$, the differing numerical support is
primarily caused by the exponential behavior of the integrand on the
side of the interval of integration where $e^{-g(\theta)} \to 0$.
Making use of this observation, for $\alpha \leq 1$ 
solving for $\theta_\eps$ such that
$g(\theta_\eps) e^{-g(\theta_\eps)} < \eps$ allows the integral
  in~\eqref{eq_statphas}  to be approximated as:
\begin{equation}
f(x; \alpha,\beta) = \frac{\alpha}{\pi |\alpha - 1| } 
\frac{1}{(x - \zeta)} \int_{\theta_\eps}^{\pi/2}
g(\theta;x,\alpha,\beta) \, 
e^{-g(\theta;x,\alpha,\beta) } \, d\theta,
\end{equation}
and for $\alpha > 1$:
\begin{equation}
f(x; \alpha,\beta) = \frac{\alpha}{\pi |\alpha - 1| } 
\frac{1}{(x - \zeta)} \int_{-\theta_0}^{\theta_\eps}
g(\theta;x,\alpha,\beta) \, 
e^{-g(\theta;x,\alpha,\beta) } \, d\theta.
\end{equation}
A change of variable in these integrals can translate the interval of
integration to $[0,1]$. If the generalized Gaussian quadrature
construction procedure is applied to these formulae, we also observe a
reduction in the number of nodes and weights required. For example, in
the symmetric case, solving for $g(\theta_\eps) = 40$ (which yields
double-precision decay) the rank of the matrix~$\mtx{U}$
in~\eqref{eq_umat} decreased from 290 to 68 for $x \in [10^{-5},30]$
and $\alpha \in [.6,.8]$

In practice, the transformation of the stationary phase integral does
not prove to be efficient because it relies on an initial zero-finding
procedure to construct transformation in $\theta$ depending on 
each parameter $x$, $\alpha$, $\beta$.
In contrast, the change of variables in~\eqref{eq_changevar} merely
requires evaluating $\log$.
Similar changes of variables can be used to simplify the construction
of quadratures for evaluating gradients of $f$.

\subsection{Alternatives to generalized Gaussian quadrature}
\label{sec_altggquad}

A number of quadrature
techniques which are particularly effective for highly oscillatory 
integrands have been developed
relatively recently. See~\cite{iserles_2006, olver_2008, olver_2010}
for an informative overview of such methods. 
Two notable examples which we will discuss here are
Filon-type and Levin-type methods.
These techniques are applicable to integrals of the form
\begin{equation}
\label{eq_osci_form}
\int f(t) \, e^{i \omega g(t)} \, dt,
\end{equation} 
where $f$ and $g$ are smooth, non-oscillatory 
functions, and $\omega$ is a scalar. The function~$g$ is 
called the {\it oscillator}.

The integrand of~\eqref{eq_inverse_Fourier} is oscillatory, 
and therefore one could consider applying either Filon- or Levin-type
quadrature schemes 
 instead of the generalized Gaussian quadrature rules.
 Unfortunately, for general parameters ranges,
 generalized Gaussian quadratures are likely to be the most efficient
 schemes. We briefly justify this statement with a discussion of Filon
 and Levin methods.

Filon-type quadratures are interpolatory quadrature rules.
That is, they approximate the function $f$ with 
a set of functions $\psi_k$ for which an analytical solution of the integral
\begin{equation}
\label{eq_filon}
\mu_k = \int \psi_k(t) \, e^{i\omega g(t)} \, dt
\end{equation}
exists.
This poses a problem for the application of Filon-type quadratures to
the integral~\eqref{eq_inverse_Fourier}.
Indeed, note that~\eqref{eq_inverse_Fourier}
can be written in the form of~\eqref{eq_osci_form} 
by setting
\begin{equation}
  \begin{aligned}
    f(t) &= e^{-t^\alpha}, \\
    g(t) &= (x-\zeta) t + \zeta t^\alpha.
  \end{aligned}
\end{equation}
Since~$g$ depends on~$\alpha$, $\beta$, and $x$, 
the integrals~$\mu_k$ have to be recalculated for every evaluation of
a stable density with differing parameters.
In contrast, the generalized Gaussian quadratures we derived
are applicable for a wide range of parameters
(see Figure~\ref{fig_regions}).

Turning to Levin-type methods, they can be
illustrated with the
following observation~\cite{olver_2010}.  Let $u(t)$ be a function
which satisfies
\begin{equation}
\label{eq_levin_ansatz}
\frac{d}{dt} \left( u(t) \, e^{i\omega g(t)} \right) = f(t) \, e^{i\omega g(t)} .
\end{equation}
We then have
\begin{equation}
\label{eq_levin_int}
\int_a^b f(t) \, e^{i\omega g(t)} dt = u(b) \, e^{i\omega g(b)}
- u(a) \, e^{i\omega g(a)}.
\end{equation}
From~\eqref{eq_levin_ansatz}, we can also
derive the differential equation
\begin{equation}
\label{eq_levind}
u' + i \omega \, g' \, u = f,
\end{equation}
where $u'$ and $g'$ denote differentiation with respect to $t$.
The problem of computing  an oscillatory integral has therefore been
converted to that of solving a first-order, linear differential
equation on the interval $[a,b]$.

Levin-type methods will require the solution of this ODE \emph{every
  time} an integral has to be evaluated. Even with high-order
convergent ODE solvers, these methods are unlikely to beat generalized
Gaussian quadrature methods in terms of floating-point operations
(after suitable offline pre-computations).
Levin-type methods could be applicable to the parameter range
$.9 < \alpha < 1.1$, $\beta \neq 0$, where modestly-sized
generalized Gaussian quadratures are not available.
 For example, if a Chebyshev spectral method is used to
solve~\eqref{eq_levind}, 
numerical experiments indicate that the 
condition number of the system can reach~$\sim 10^6$ before
the solution $u$ can be fully resolved (requiring $\sim 1000$
Chebyshev terms).
Therefore, a significant loss of accuracy with the Chebyshev spectral 
method is likely.
Indeed, sometimes only 9 significant digits are achieved with this approach.
A more specialized solver would have to be developed to 
solve~\eqref{eq_levind} efficiently and with higher accuracy.

Levin-type methods could also be used for integral representations
of the partial derivatives of stable densities,
for which no asymptotic expansion is available.
However, it is likely to be cheaper to form a Chebyshev interpolant
of the density as a function of the parameter, and then differentiate
the series (i.e. perform 2D interpolation and spectral
differentiation). This will achieve higher accuracy than a 
finite difference scheme, with a slightly higher computational cost.

\section{Algorithm \& Numerical examples}
\label{sec_algorithm}

In the following, we will describe the details of our algorithm.  In
particular, we detail which formula or quadrature should be used
depending on values of the parameters $x$, $\alpha$, $\beta$.  We
begin with a comparison of the benefits of the two integral
representations given by~\eqref{eq_inverse_Fourier}
and~\eqref{eq_statphas}.

\subsection{Choosing an integral representation}
It has become clear after several numerical experiments that the
stationary phase integral~\eqref{eq_statphas}, while seemingly simpler
to evaluate than~\eqref{eq_inverse_Fourier}, carries several
disadvantages.  Namely, it cannot be used to reliably evaluate~$f$
when~$x \sim \zeta$ (the mode of the distributions in the symmetric
case),~$\alpha \sim 1$, or for very large $x$. Furthermore, its
partial derivatives suffer from the same deficiencies and have rather
unwieldy forms.  Lastly, the rank-reducing technique of
Section~\ref{sec_rank} is not as effective when applied 
to~\eqref{eq_statphas}. This results in quadratures of much larger
sizes when compared to those for~\eqref{eq_inverse_Fourier}.

In contrast,~\eqref{eq_inverse_Fourier} has only one of the
aforementioned deficiencies: it cannot be evaluated efficiently when
$\alpha \sim 1$ in the asymmetric case.  Still, it is important to
point out that~\eqref{eq_inverse_Fourier} can be easily evaluated at
$\alpha \sim 1$ in the symmetric case.

The stationary~phase form~\eqref{eq_statphas} does have two advantages
over~\eqref{eq_inverse_Fourier}.  First, it is well behaved for
intermediate to large $x$, whereas~\eqref{eq_inverse_Fourier} becomes
very oscillatory.  Second, it can be evaluated for $\alpha < 0.5$,
whereas the relevant interval of integration
of~\eqref{eq_inverse_Fourier} grows rather fast as $\alpha \to 0$.
However, the series expansion~\eqref{eq_series_at_infinity} is a much
more efficient means of evaluation in these regimes.  This limits the
usefulness of the stationary phase integral for our purpose.

As a result, the only integral representation of the density
 we use in our algorithm
is given by~\eqref{eq_inverse_Fourier}.  One consequence of this
choice is that we do not need to use the series expansion around
$x = \zeta$ given in~\eqref{eq_series_at_zero}, as the integral is
well behaved there.
For similar reasons, we use~\eqref{eq_cdf}
to compute $F$, and the integral
representations for the gradient of $f$
given in~\eqref{eq_partials}.

\begin{figure}[!t]
  \centering
  \begin{subfigure}{.45\linewidth}
    \centering
    \includegraphics[width=.95\linewidth]{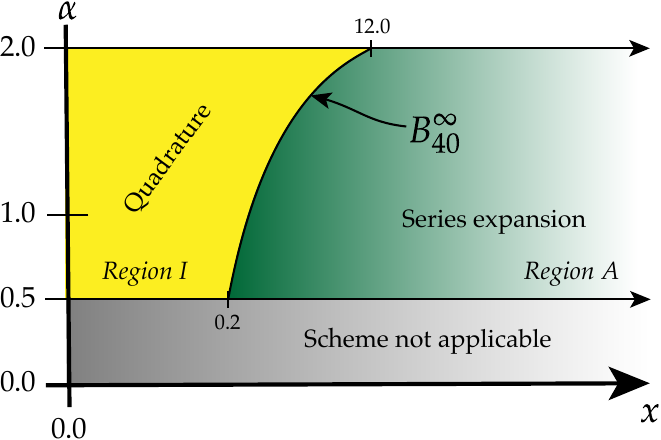}
    \caption{The symmetric case, $\beta=0$.}
    \label{fig_symmetric}
  \end{subfigure}
  \qquad
  \begin{subfigure}{.45\linewidth}
    \centering
    \includegraphics[width=.95\linewidth]{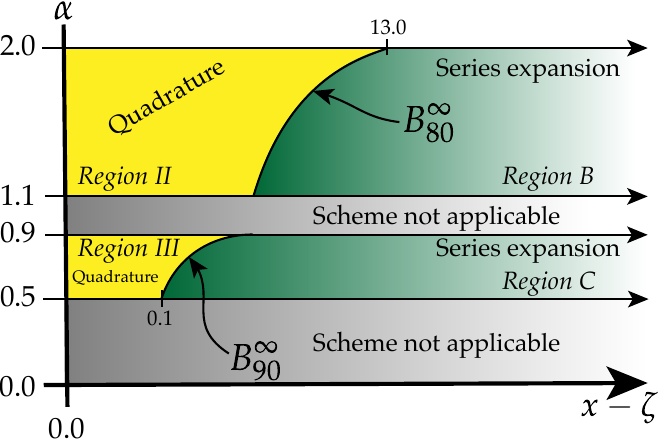}
    \caption{The asymmetric case, $\beta\neq 0$.}
    \label{fig_asymmetric}
  \end{subfigure}
  \caption{Regions of validity for generalized Gaussian quadrature
    rules and series approximations for the evaluation of
    $f(x;\alpha,\beta)$. Asymptotic expansions are used for
  extreme values of $x$ and generalized Gaussian quadrature
  routines are able to fill in large regions of the remaining space.}
  \label{fig_regions}
\end{figure}

\subsection{The symmetric case $\beta = 0$}
\label{sec_symmetric}

We first provide some numerical examples of the accuracy and
efficiency of evaluating the symmetric densities
\begin{equation}\label{eq_simplesymmetric}
f(x;\alpha,0) = \frac{1}{\pi} \int_0^\infty  \cos (xt)
\, e^{-t^\alpha} \, dt.
\end{equation}
We restrict our attention to values of $f$ for $\alpha \geq 0.5$ for
two reasons.  First, when $\alpha$ is much smaller than $0.5$, and $x$
is close to but not equal to $\zeta$, existing numerical schemes for
integral representations and series expansions require prohibitive
computational cost to achieve reasonable accuracy.  And second, the
applications for modeling with stable laws with such small values of
the stability parameter $\alpha$ seem to be very rare.  Nevertheless,
it should be pointed out that when $\alpha < 0.1$, the
series~\eqref{eq_series_at_infinity} with $n_\infty = 128$ terms is
accurate to double precision for $x - \zeta \geq 10^{-16}$.  At
$x = \zeta$, the first term of the series~\eqref{eq_series_at_zero}
can be used to obtain an accurate value of the density.  Therefore, in
this extreme regime, effective numerical evaluations of stable laws is
possible using the series expansions alone.

\begin{table}[!b]
  \begin{center}
    \caption{Symmetric ($\beta=0$) stable density evaluation for
      $\alpha \in [.5, 2.0]$.}
    \label{tab_symmetric}
      \begin{tabular}{|c|c|c|c|c|}
        \hline
        & $n_{GGQ}$  & $n_{\infty}$ & $\max$ err  \\ \hline
        $f$ & $43$ & $42$ & $5 \se{-14}$ \\ \hline
        $\partial_x f$ & $44$ & $42$ & $9 \se{-14}$ \\ \hline
        $\partial_\alpha f$ & $49$ & $42$ & $1 \se{-13}$ \\ \hline
        $F$ & $39$ & $42$ & $9 \se{-14}$ \\ \hline
      \end{tabular}
  \end{center}
\end{table}

We now move to a description of our evaluation scheme.
For a particular value of $\alpha \geq 0.5$, if \mbox{$x \leq B_{40}^\infty$}
(this corresponds to Region~I in Figure~\ref{fig_symmetric}), we use a
43-point generalized Gaussian quadrature to evaluate the above 
integral. If $x > B_{40}^\infty$, we use series 
expansion~\eqref{eq_series_at_infinity}. 
The number of terms in the
series expansion was chosen (experimentally, as a precomputation)
to roughly equal the number of nodes in the optimized quadrature.
A similar method is used for the computation 
of $F$ and the gradient of $f$.
However, there is a notable difference in the computation 
of~$\partial_\alpha f$ as it does not permit a convenient series expression,
as noted in Section~\ref{sec_derivatives}. 
Instead, we use a finite difference scheme applied to
the series~$\eqref{eq_series_at_infinity}$  to 
compute~$\partial_\alpha f$ when~$x > B_{43}^\infty$.
The accuracy of finite difference schemes depends on the particular
 scheme used.
In practice, a two-point finite difference is accurate to 
about $10^{-6}$, while
 a fourth order scheme is accurate to about~$10^{-10}$.
The fourth order scheme is listed in the Appendix~\ref{app_grad}.

Accuracy results for $f$ and its gradient
are reported in Table~\ref{tab_symmetric}. The columns are as follows:
\begin{quote}
  $n_{GGQ}$: \tabto{.75in} the number of nodes in the generalized Gaussian
  quadrature scheme,\\ 
  $n_{\infty}$: \tabto{.75in} the number of terms used for the 
  series~\eqref{eq_series_at_infinity}, and\\
  $\max$ error: \tabto{.75in} the maximum absolute $L_\infty$ error relative to
  adaptive integration.
\end{quote}

The accuracy results were
obtained by testing our quadrature scheme
against an adaptive integration evaluate of~\eqref{eq_inverse_Fourier}
at 100,000
randomly chosen points in the $x\alpha$-plane for $x\in
[0, B_{n_\infty}^\infty]$ and $\alpha \in [0.5,2.0]$.
We should note that all results are reported in \emph{absolute}
precision. When evaluating integrals with arbitrarily
sign-changing integrands via quadrature, 
if the integral is of size $\delta$ then it is likely that $\mathcal
O(|\log\delta|)$ digits will be lost in \emph{relative precision} due to
the cancellation effect inherent in floating-point arithmetic.
Table~\ref{tab_quads} contains the 43-point quadrature for evaluating
stable densities in Region~I of Figure~\ref{fig_symmetric}.

  \begin{table}
    \begin{center}
      \caption{Nodes and weights for computing the integral
        in~\eqref{eq_simplesymmetric} in Region~I of
        Figure~\ref{fig_symmetric}. Note that the change of variables
        discussed in Section~\ref{sec_rank} must be used before applying
        the quadrature.}
      \label{tab_quads}
      \begin{tabular}{|c|c|c|} \hline
        $j$ & $t_j$  & $w_j$ \\ \hline
$ 1 $ & $ 3.8153503841778930\se{-08} $ & $ 1.9462166165433782\se{-07} $ \\ \hline 
$ 2 $ & $ 1.8621751229398742\se{-06} $ & $ 5.6557228645853394\se{-06} $ \\ \hline 
$ 3 $ & $ 2.3548989111566051\se{-05} $ & $ 5.0123980914007912\se{-05} $ \\ \hline 
$ 4 $ & $ 1.4796873542253231\se{-04} $ & $ 2.3484191896467563\se{-04} $ \\ \hline 
$ 5 $ & $ 5.9719633529811916\se{-04} $ & $ 7.3189687338231666\se{-04} $ \\ \hline 
$ 6 $ & $ 1.7776065804175705\se{-03} $ & $ 1.7238717892356147\se{-03} $ \\ \hline 
$ 7 $ & $ 4.2473152693930051\se{-03} $ & $ 3.3181618633886167\se{-03} $ \\ \hline 
$ 8 $ & $ 8.6062904061371317\se{-03} $ & $ 5.4843557934027244\se{-03} $ \\ \hline 
$ 9 $ & $ 1.5348863951004616\se{-02} $ & $ 8.0460517169448388\se{-03} $ \\ \hline 
$ 10 $ & $ 2.4742939762206897\se{-02} $ & $ 1.0741992568943348\se{-02} $ \\ \hline 
$ 11 $ & $ 3.6794136418563730\se{-02} $ & $ 1.3324899124821651\se{-02} $ \\ \hline 
$ 12 $ & $ 5.1299788260226145\se{-02} $ & $ 1.5632319416985788\se{-02} $ \\ \hline 
$ 13 $ & $ 6.7944092105184303\se{-02} $ & $ 1.7598001767457079\se{-02} $ \\ \hline 
$ 14 $ & $ 8.6382423526857308\se{-02} $ & $ 1.9224720756886148\se{-02} $ \\ \hline 
$ 15 $ & $ 1.0629323929619865\se{-01} $ & $ 2.0550906564542663\se{-02} $ \\ \hline 
$ 16 $ & $ 1.2740084223127754\se{-01} $ & $ 2.1626845166204386\se{-02} $ \\ \hline 
$ 17 $ & $ 1.4948000254495675\se{-01} $ & $ 2.2501767869303416\se{-02} $ \\ \hline 
$ 18 $ & $ 1.7235168105825832\se{-01} $ & $ 2.3218324218440851\se{-02} $ \\ \hline 
$ 19 $ & $ 1.9587547846015377\se{-01} $ & $ 2.3811106669646236\se{-02} $ \\ \hline 
$ 20 $ & $ 2.1994170091684220\se{-01} $ & $ 2.4307093498802106\se{-02} $ \\ \hline 
$ 21 $ & $ 2.4446430088367060\se{-01} $ & $ 2.4726814975746716\se{-02} $ \\ \hline 
$ 22 $ & $ 2.6937507294734536\se{-01} $ & $ 2.5085627984821550\se{-02} $ \\ \hline 
$ 23 $ & $ 2.9461905048621601\se{-01} $ & $ 2.5394814769833289\se{-02} $ \\ \hline 
$ 24 $ & $ 3.2015086713296453\se{-01} $ & $ 2.5662404915729992\se{-02} $ \\ \hline 
$ 25 $ & $ 3.4593177859400515\se{-01} $ & $ 2.5893662976614856\se{-02} $ \\ \hline 
$ 26 $ & $ 3.7192698736533930\se{-01} $ & $ 2.6091208338375568\se{-02} $ \\ \hline 
$ 27 $ & $ 3.9810287487972756\se{-01} $ & $ 2.6254650270947675\se{-02} $ \\ \hline 
$ 28 $ & $ 4.2442340488910107\se{-01} $ & $ 2.6379218475411595\se{-02} $ \\ \hline 
$ 29 $ & $ 4.5084450818929106\se{-01} $ & $ 2.6452938570694140\se{-02} $ \\ \hline 
$ 30 $ & $ 4.7730398807466573\se{-01} $ & $ 2.6449899946836126\se{-02} $ \\ \hline 
$ 31 $ & $ 5.0370157776242630\se{-01} $ & $ 2.6317157280117857\se{-02} $ \\ \hline 
$ 32 $ & $ 5.2986292621392794\se{-01} $ & $ 2.5956236923454400\se{-02} $ \\ \hline 
$ 33 $ & $ 5.5549151318191370\se{-01} $ & $ 2.5233135355892482\se{-02} $ \\ \hline 
$ 34 $ & $ 5.8023336057818919\se{-01} $ & $ 2.4253858975533026\se{-02} $ \\ \hline 
$ 35 $ & $ 6.0420106522936246\se{-01} $ & $ 2.3883208046443095\se{-02} $ \\ \hline 
$ 36 $ & $ 6.2845118361063135\se{-01} $ & $ 2.4788564380040439\se{-02} $ \\ \hline 
$ 37 $ & $ 6.5391666500166423\se{-01} $ & $ 2.6135655855085593\se{-02} $ \\ \hline 
$ 38 $ & $ 6.8067763680759019\se{-01} $ & $ 2.7386023987669407\se{-02} $ \\ \hline 
$ 39 $ & $ 7.0883363435562430\se{-01} $ & $ 2.9079584045104245\se{-02} $ \\ \hline 
$ 40 $ & $ 7.3935214962210505\se{-01} $ & $ 3.2403729259281477\se{-02} $ \\ \hline 
$ 41 $ & $ 7.7501382927296592\se{-01} $ & $ 3.9683359210637488\se{-02} $ \\ \hline 
$ 42 $ & $ 8.1983271443438077\se{-01} $ & $ 5.0313579393503942\se{-02} $ \\ \hline 
$ 43 $ & $ 8.7653187131388799\se{-01} $ & $ 6.3807406535572972\se{-02} $ \\ \hline 
      \end{tabular}
    \end{center}
  \end{table}

\begin{figure}[t]
  \begin{center}
    \begin{subfigure}{.45\linewidth}
      \begin{center}
        \includegraphics[width=.95\linewidth]{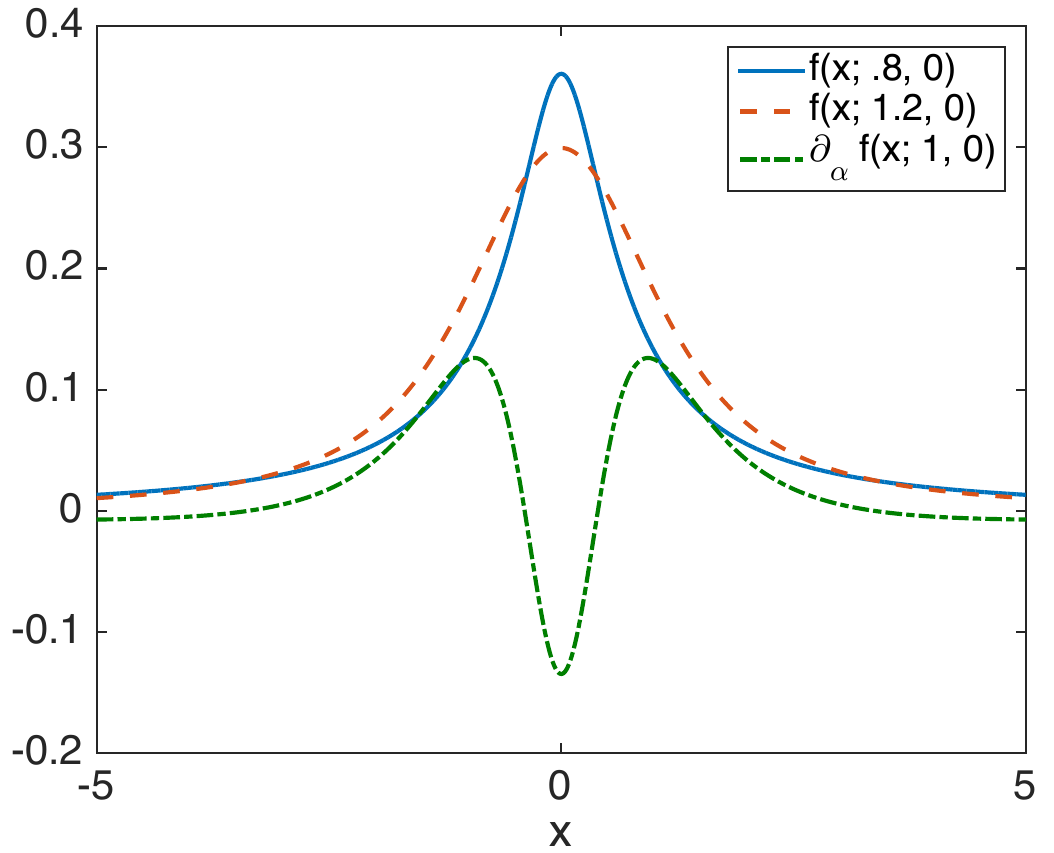} 
      \end{center}
      \caption{Varying $\alpha$.}
      \label{fig_alpha_effects}
    \end{subfigure}
    \quad
    \begin{subfigure}{.45\linewidth}
      \begin{center}
        \includegraphics[width=.95\linewidth]{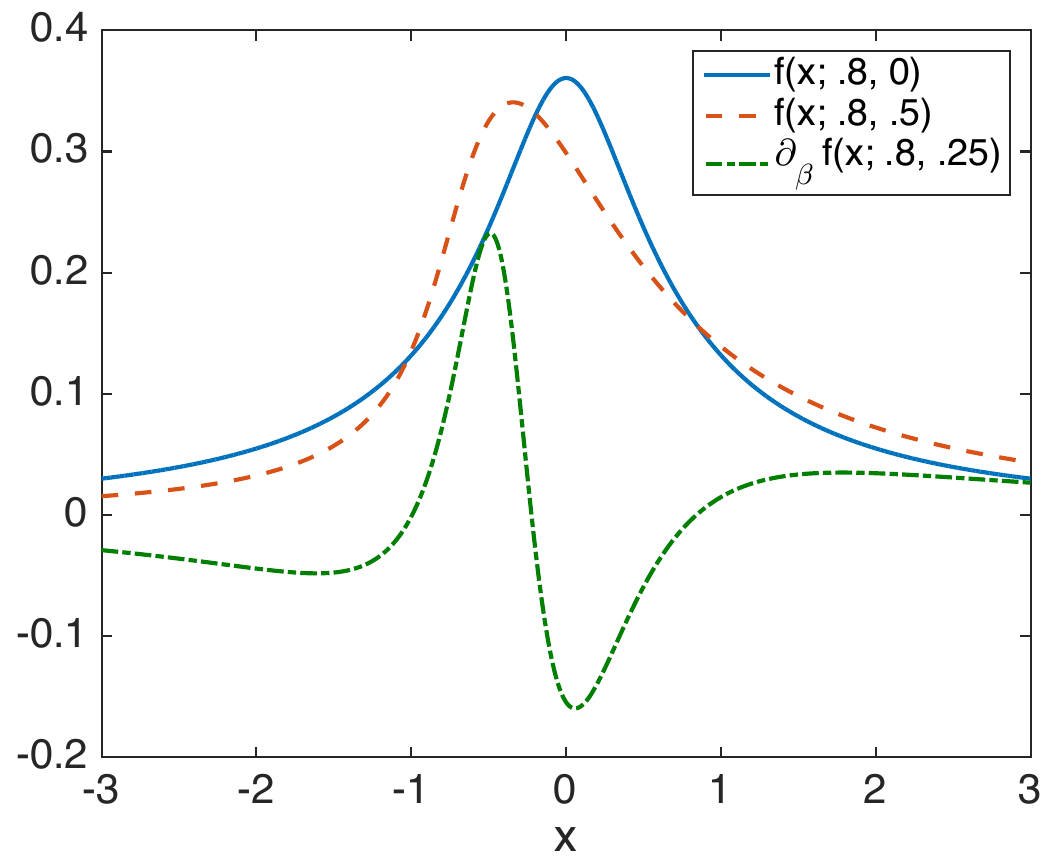} 
      \end{center}
      \caption{Varying $\beta$.}
      \label{fig_beta_effects}
    \end{subfigure}
  \end{center}
  \caption{Effect of changing parameters of the density 
    $f(x;\alpha,\beta)$, along with partial derivatives.}
\end{figure}

We should note that while it is possible to construct more efficient
quadratures for smaller regions of the $x\alpha$-plane, namely for
$\alpha>1$ (distributions with finite expectation), it is useful to
obtain a single global quadrature value in a lone region. As shown in
Figure~\ref{fig_alpha_effects}, small changes in $\alpha$ induce
equivalently small changes in the density (a rather low-rank update in
Fourier-space).  In particular applications with restricted stability
parameters, it may be prudent to construct even more efficient
quadratures.  There are several parameter combinations or ranges that
might benefit from specialized quadrature.  For example, the
Holtsmark distribution ($\alpha = 1.5$, $\beta = 0$) occurs in
statistical investigations of gravity~\cite{chand_1943,
  chavanis_2009}.  The methods of this paper can be applied
to compute this distribution, and others, very efficiently.

\subsection{The asymmetric case $\beta \neq 0$}
\label{sec_asymmetric} 

\begin{table}[!t]
  \begin{center}
    \caption{Asymmetric ($\beta \neq 0$) stable density evaluation for
      $\alpha \in [0.5,0.9]$ and $\alpha \in [1.1,2.0]$.}
    \begin{subtable}{.45\linewidth}
      \begin{center}
      \caption{$\alpha \in [.5, .9]$ }
      \label{tab_asym1}
      \begin{tabular}{|c|c|c|c|}  \hline
        & $n_{GGQ}$  & $n_{\infty}$ & $\max$ err \\ \hline
        $f$ & $94$ & $90$ & $5 \se{-14}$  \\ \hline
        $\partial_x f$ & $110$ & $90$ & $1 \se{-13}$ \\ \hline
        $\partial_\alpha f$ & $113$ & $90$ & $9 \se{-14}$ \\ \hline
        $\partial_\beta f$ & $109$ & $90$ & $5 \se{-14}$ \\ \hline
        $F$ & $181$ & $90$ & $1 \se{-8}$ \\ \hline
      \end{tabular}
      \end{center}
    \end{subtable} \qquad
    \begin{subtable}{.45\linewidth}
      \begin{center}
      \caption{$\alpha \in [1.1,2.0]$}
      \label{tab_asym2}
      \begin{tabular}{|c|c|c|c|}
        \hline
        & $n_{GGQ}$  & $n_{\infty}$ & $\max$ err \\ \hline
        $f$ & $86$ & $80$ & $2 \se{-14}$ \\ \hline
        $\partial_x f$ & $96$ & $80$ & $2 \se{-14}$\\ \hline
        $\partial_\alpha f$ & $98$ & $80$ & $9 \se{-14}$ \\ \hline
        $\partial_\beta f$ & $93$ & $80$ &  $4 \se{-14}$ \\ \hline
        $F$ & $88$ & $80$ & $1 \se{-14}$ \\ \hline
      \end{tabular}
      \end{center}
    \end{subtable}
  \end{center}
\end{table}

In the asymmetric case, $\beta\neq 0$, we first change variables and
evaluate the densities at locations relative to: $x-\zeta$.  This
ensures that the densities are continuous in all parameters.  As in
the symmetric case, we restrict our attention to densities with
$\alpha \geq 0.5$.  Furthermore, due to difficulties in the integral
and series formulations near $\alpha = 1$, we partition the $\alpha$
space into two regions: $[0.5,0.9]$ and $[1.1,2.0]$.  For values of
$\beta \neq 0$, $|\zeta| \to \infty$ as $\alpha \to 1$.  This is the
main mode of failure for both integral
representations~\eqref{eq_inverse_Fourier} and~\eqref{eq_statphas}
near~$\alpha = 1$.  As a consequence, the integrand
in~\eqref{eq_inverse_Fourier} becomes highly oscillatory for even
small values of $x-\zeta$, and~\eqref{eq_statphas} becomes spiked, as seen
in Figure~\ref{fig_spiked}.  Quadrature techniques developed to deal
with highly oscillatory integrands may be applicable in this regime,
and will be investigated in future work.

For calculating asymmetric densities, 
the parameter space is partitioned in the following manner:
for all $0 \leq x - \zeta \leq B^\infty_{n_\infty}$, 
the densities are calculated via a generalized Gaussian
quadrature scheme for the integral~\eqref{eq_inverse_Fourier}.
For $x - \zeta > B^\infty_{n_\infty}$,  
the series expansion~\eqref{eq_series_at_infinity} is used.
As mentioned previously, we have not obtained a 
convenient series representation of~$\partial_\alpha f$ and~$\partial_\beta f$. 
Similar to the computation of~$\partial_\alpha f$ in the symmetric case, 
we use finite differences to approximate values of~$\partial_\alpha f$
and~$\partial_\beta f$ whenever $x > B_{n_\infty}^\infty$.
Depending on the finite difference scheme used, 
this may lead to reduced accuracy compared to the quadrature 
method for the computation of~$f$ and~$\partial_x f$.
Similar accuracy reports to those for the symmetric densities 
are contained in  Tables~\ref{tab_asym1} and~\ref{tab_asym2}.
Notably, the quadrature rule for $F$
in the regime $\alpha \in [.5, .9]$ is less accurate
and has more nodes/weights than for the other functions.
This is due to the fact that the integrand in~\eqref{eq_cdf} for $F$
is singular at $t=0$ in the asymmetric case. 
As a consequence, designing highly-accurate quadrature
rules for~\eqref{eq_cdf} without using quadruple 
precision calculations is not possible.
This issue will be investigated in future work.
The corresponding quadrature rules are available for download
at~\thewebsite.

\subsection{Efficiency of the method}

To test the efficiency of our method, we compare our implementation of
the density function evaluation to two different 
implementations based on adaptive quadrature.  
All codes are written in \textsc{Matlab}.  The first
implementation simply applies \textsc{Matlab}'s \texttt{integral}
function to
the oscillatory integral~\eqref{eq_inverse_Fourier}.  Note that this
function can be called in a vectorized manner by adjusting the
\texttt{ArrayValued} argument.  
Without this adjustment, the computations below are about an order of
magnitude slower. (I.e. $t_{AQ1}$ and $t_{AQ2}$ are roughly 10 times
as large.)
The second implementation mimics the approach
that was previously taken to compute the stationary 
phase integral~\eqref{eq_statphas}.  
Namely, it first locates the peak of the
integrand using \textsc{Matlab}'s intrinsic \texttt{fzero} function,
and subsequently applies \texttt{integral}  on the
two subintervals created by splitting the original interval of
integration at the peak of the integrand.
 
\begin{table}[t]
  \begin{center}
    \caption{Timings for density evaluations}
    \label{tab_timings}
      \begin{tabular}{|c|c|c|c|c|c|}
        \hline
       	 $t_{GQ}$ & $t_{GQ}$ ($\beta = 0$)  & $t_{AQ1}$ & $t_{AQ2}$ \\ \hline
         $0.003$ sec & $0.001$ sec & $0.3$ sec & $2$ sec \\ \hline
      \end{tabular}
  \end{center}
\end{table}

The validation test proceeds as follows. 
First,~$\alpha$ and~$\beta$
are chosen randomly in the permissible parameter ranges.  Then,
10,000 uniformly random $x$ are generated such that
$0 \leq x-\zeta \leq 20$.  Thereafter, we record the wall-clock time
each method takes to calculate the stable density at all 10,000
points.  For our tests, we require the absolute accuracy of the
adaptive schemes to be $10^{-10}$.  The results are reported in 
Table~\ref{tab_timings}.
  The columns of the table are:
\begin{quote}
$t_{GQ}$: \tabto{.75in} the time taken by our scheme to compute the
density at all points,\\
$t_{AQ1}$: \tabto{.75in} time taken by the first adaptive scheme
outlined above, and \\
 $t_{AQ2}$: \tabto{.75in} the time taken by the second adaptive scheme
outlined above.
\end{quote}
  We also report a timing for the symmetric case ($\beta = 0$) 
for our scheme, since it uses a quadrature separate from the one in
the asymmetric case.
The test was performed on a MacBook Pro with a 2.4~GHz Intel Core 
i7 and 8 GB 1333 MHz DDR3 RAM.  As one can see, our
scheme outperforms the adaptive ones by at least two orders of
magnitude.

\section{Conclusions}
\label{sec_conclusions}

In this work, we have presented efficient quadrature schemes and
series expansions for numerically evaluating the densities, and
derivatives thereof, associated with what are known as stable
distributions.  The quadratures are of generalized Gaussian type, and
were constructed using a non-linear optimization procedure. The series
expansions were obtained straightforwardly from integral
representations, but seem to have not been previously presented in the
computational statistics literature.  The methods of this paper are
quite efficient, and easily vectorizable.  This is in contrast to
existing schemes for evaluating these integrals, which were
predominately based on adaptive integration -- which cannot take full 
advantage of vectorization schemes due to varying  depths of recursion.


Furthermore, while the quadratures that we constructed are (nearly)
optimal with respect to the number of nodes and weights required, they
do not obtain full double precision accuracy ($\sim 10^{-16}$). We
often only achieve absolute accuracies of $12$ or $13$ digits. While
some of the precision loss is due to merely roundoff error in summing
the terms in the quadrature, some of the loss of accuracy is due to
solving the ill-conditioned linearization of the quadrature problem.
The accuracy lost due to this aspect of the procedure could be
recovered if the quadrature generation codes were re-written using
quadruple precision arithmetic instead of double precision. In most
cases the accuracy we obtained is sufficient for general use, but we
are investigating a higher precision procedure for constructing the
quadrature rules.

The schemes presented in this paper still fail to thoroughly address
the evaluation of the density function (and gradient and CDF)
 for values of
$\alpha \approx 1$ in the asymmetric case. One could, however, perform a large-scale
precomputation in extended precision in order to tabulate these
densities for various values of $x$ and $\beta$, store the results,
and later interpolate to other values. This approach was beyond the
scope of this work. 
This approach was used for maximum likelihood estimation
in~\cite{Nolan_2001}. Unless chosen very carefully, 
a rather large number of interpolation nodes are necessary to 
achieve high accuracy, and each function ($f, \nabla f, F$) 
has to be tabulated separately.
We are actively investigating approaches to fill
in this gap in the numerical evaluation of  the
density (and gradient and CDF).

A software package written in \textsc{Matlab} for computing stable
densities, their gradients, and distribution functions
using the algorithms of this paper is available
at~\thewebsite, and will be continually updated as we improve the
efficiency and accuracy of existing evaluations, and include
additional capabilities.

\section*{Acknowledgments}
The authors would like to thank Sinan G\"unt\"urk
and Margaret Wright for several useful
conversations.

\begin{appendix}

\section{Gradients of series expansions}
\label{app_grad}

Here we provide formulae for the derivatives of the series expansions
presented in Section~\ref{sec_derivatives}.
From~\eqref{eq_series_at_zero},
\begin{equation}
\label{eq_series_xder_at_zero}
\partial_x f(x;\alpha, \beta) = \frac{1}{\alpha \pi} \sum_{k=0}^\infty \frac{ \Gamma( \frac{k+2}{\alpha} )}{ \Gamma( k ) } 
			(1+\zeta^2)^{ - \frac{k+2}{ 2\alpha}} 
			\sin \left( \left( \pi/2 + (\arctan \zeta)/\alpha \right)
			(k+2) \right) (x - \zeta)^{k}.
\end{equation}
Using an error bound analogous to the ones in Section~\ref{sec_asymp},
we have that
\begin{equation}
|x-\zeta| \leq C_n^0(\alpha) := \left[ \eps \alpha \pi 
		(1+\zeta^2)^{\frac{n+2}{2\alpha}} 
		\frac{\Gamma(n)} {\Gamma( \frac{n+2}{\alpha})}
		\right]^{1/n}.
\end{equation}

By differentiating \eqref{eq_series_at_infinity}, we attain
\begin{equation}
\label{eq_series_xder_at_infinity}
\begin{aligned}
  \partial_x f(x,\alpha, \beta) = \frac{\alpha}{\pi} \sum_{k=1}^\infty
  (-1)^{k} \frac{ (\alpha k + 1) \Gamma(\alpha k)}{\Gamma(k)}
  (1+\zeta^2)^{k/2} \sin((\pi \alpha - \arctan \zeta) k)
  (x-\zeta)^{-\alpha k-2},
\end{aligned}
\end{equation}
whose radius of convergence to precision $\eps$ we estimate by
\begin{equation}
  |x-\zeta| \geq C_{n-1}^\infty(\alpha) :=  \left[ \frac{\alpha}{\pi \eps} 
    (1+\zeta^2)^{\frac{n}{2} }
    \frac{ (\alpha n + 1)\Gamma(\alpha n)}{\Gamma( n ) } 
  \right]^{1/(\alpha n-2)}.
\end{equation}

For the parameter ranges where there is no convenient formulation of
the derivatives, we can use a finite difference approximation of the
form
\begin{equation}
  \partial_x f(x) = \frac{ -f(x+2h) + 8 f(x+h) 
- 8 f(x-h) + f(x-2h) }{12h} + O(h^4).
\end{equation}

\end{appendix}

\bibliographystyle{abbrv}
\bibliography{../preprint}

\begin{thebibliography}{10}

\bibitem{bremer_2010b}
J.~Bremer, Z.~Gimbutas, and V.~Rokhlin.
\newblock A nonlinear optimization procedure for generalized {G}aussian
  quadratures.
\newblock {\em SIAM J. Sci. Comput.}, 32:1761--1788, 2010.

\bibitem{chand_1943}
S.~Chandrasekhar.
\newblock Stochastic problems in physics and astronomy.
\newblock {\em Rev. Mod. Phys.}, 15(1):1--89, Jan 1943.

\bibitem{chavanis_2009}
P.~Chavanis.
\newblock Statistics of the gravitational force in various dimensions of space:
  from gaussian to l{\'e}vy laws.
\newblock {\em The European Physical Journal B}, 70(3):413--433, 2009.

\bibitem{dahlquist_2003}
G.~Dahlquist and A.~Bj\"orck.
\newblock {\em Numerical {M}ethods}.
\newblock Dover, New York, NY, 2003.

\bibitem{iserles_2006}
A.~Iserles, S.~N{\o}rsett, and S.~Olver.
\newblock Highly oscillatory quadrature: The story so far.
\newblock In {\em Numerical mathematics and advanced applications}, pages
  97--118. Springer, 2006.

\bibitem{kolbig_1984}
K.~S. K\"olbig and B.~Schorr.
\newblock A program package for the {L}andau distribution.
\newblock {\em Comput. Phys. Commun.}, 31:97--111, 1984.

\bibitem{ma_1996}
J.~Ma, V.~Rokhlin, and S.~Wandzura.
\newblock {Generalized Gaussian Quadrature Rules for Systems of Arbitrary
  Functions}.
\newblock {\em SIAM J. Numer. Anal.}, 33:971--996, 1996.

\bibitem{matsui_2004}
M.~Matsui and A.~Takemura.
\newblock Some improvements in numerical evaluation of symmetric stable density
  and its derivatives.
\newblock {\em Multivar. Anal.}, 35:149--172, 2006.

\bibitem{mittnik_1993}
S.~Mittnik and S.~T. Rachev.
\newblock Modeling asset returns with alternative stable distributions.
\newblock {\em Econometric Reviews}, 12(3):261--330, 1993.

\bibitem{nikias_1995}
C.~L. Nikias and M.~Shao.
\newblock {\em Signal {P}rocessing with {A}lpha-{S}table {Distributions} and
  {A}pplications}.
\newblock Wiley, New York, NY, 1995.

\bibitem{nolan_1997}
J.~P. Nolan.
\newblock Numerical calculation of stable densities and distribution functions.
\newblock {\em {Commun. Statist.-Stochastic Models}}, 13(4):759--774, 1997.

\bibitem{Nolan_2001}
J.~P. Nolan.
\newblock {\em Maximum likelihood estimation and diagnostics for stable
  distributions.}, chapter L{\'e}vy Processes: Theory and Applications.
\newblock Birkh{\"a}user Boston, 2001.

\bibitem{nolan_2003}
J.~P. Nolan.
\newblock Modeling financial data with stable distributions.
\newblock In S.~T. Rachev, editor, {\em {Handbook of Heavy Tailed Distributions
  in Finance}}, pages 105--130. Elsevier/North Holland, New York, NY, 2003.

\bibitem{nolan_2015}
J.~P. Nolan.
\newblock {\em Stable Distributions - Models for Heavy Tailed Data}.
\newblock Birkhauser, Boston, 2015.
\newblock In progress, Chapter 1 online at
  academic2.american.edu/$\sim$jpnolan.

\bibitem{olver_2008}
S.~Olver.
\newblock {\em {Numerical Approximation of Highly Oscillatory Integrals}}.
\newblock PhD thesis, University of Cambridge, 2008.

\bibitem{olver_2010}
S.~Olver.
\newblock Fast, numerically stable computation of oscillatory integrals with
  stationary points.
\newblock {\em Bit. Numer. Math.}, 50:149--171, 2010.

\bibitem{press_2007}
W.~H. Press, S.~A. Teukolsky, and W.~T. Vetterling.
\newblock {\em {Numerical Recipes}}.
\newblock Cambridge University Press, New York, NY, 3rd edition, 2007.

\bibitem{rasmussen_2006}
C.~E. Rasmussen and C.~Williams.
\newblock {\em {Gaussian Processes for Machine Learning}}.
\newblock MIT Press, Cambridge, MA, 2006.

\bibitem{shep_1991}
N.~G. Shephard.
\newblock {From characteristic function to distribution function: A simple
  framework for the theory}.
\newblock {\em Econometric Theory}, 7:519--529, 1991.

\bibitem{teimouri2008stable}
M.~Teimouri and H.~Amindavar.
\newblock A novel approach to calculate stable densities.
\newblock In {\em Proceedings of the World Congress on Engineering}, volume~1,
  2008.

\bibitem{yarvin_1998}
N.~Yarvin and V.~Rokhlin.
\newblock Generalized {G}aussian quadratures and singular value decompositions
  of integral operators.
\newblock {\em SIAM J. Sci. Comput.}, 20:699--718, 1998.

\bibitem{zolotarev_1986}
V.~M. Zolotarev.
\newblock {\em One-dimensional stable distributions}.
\newblock American Mathematical Society, Providence, RI, 1986.

\end{thebibliography}

\end{document}